\documentclass[reqno, 10pt]{amsart}
\usepackage[english]{babel}

\usepackage{lmodern}
\makeindex
\usepackage{mathtools}
\usepackage[T1]{fontenc}	
\usepackage[utf8]{inputenc}
\usepackage{indentfirst}
\usepackage{amssymb,amsmath,amsthm,amsfonts,amscd,amstext,amsbsy}
\usepackage{mathrsfs}
\usepackage{bbm}

\usepackage{epic}
\usepackage{epsf}
\usepackage[active]{srcltx}

\usepackage{calligra}

\usepackage{empheq}
\usepackage{latexsym}				
\usepackage{graphicx,graphics,epsfig}			
\usepackage{graphpap}				
\usepackage{float}
\usepackage{xcolor}
\usepackage{color}
\usepackage{appendix}				
\usepackage{multicol}			
\usepackage{epstopdf}
\usepackage{enumerate}
\usepackage{hyperref}
\usepackage{appendix}
\usepackage[shortlabels]{enumitem}
\usepackage{comment}

\theoremstyle{plain}
\newtheorem{theorem}{Theorem}[section]

\newtheorem{lemma}[theorem]{Lemma}

\newtheorem{proposition}[theorem]{Proposition}

\newtheorem{remark}[theorem]{Remark}

\DeclareUnicodeCharacter{221E}{\ensuremath{\infty}}

\DeclareMathOperator{\R}{\mathbb{R}}
\DeclareMathOperator{\C}{\mathbb{C}}
\DeclareMathOperator{\N}{\mathbb{N}}
\DeclareMathOperator{\Z}{\mathbb{Z}}
\DeclareMathOperator{\T}{\mathbb{T}}

\def\S{\mathbb{S}}
\def\T{\mathbb{T}}

\def\id{\operatorname{id}}
\def\div{\operatorname{div}}
\def\dim{\operatorname{dim}}
\def\Re{\operatorname{\mathfrak{Re}}}

\def\cA{\mathcal{A}}
\def\cAd{\mathcal{A}_{\delta}}
\def\cAad{\mathcal{A}_{\alpha,\delta}}
\def\cAed{\mathcal{A}_{e,\delta}}
\def\cAud{\mathcal{A}_{1,\delta}}
\def\cB{\mathcal{B}}
\def\cBd{\mathcal{B}_{\delta}}
\def\cBed{\mathcal{B}_{e,\delta}}
\def\cBad{\mathcal{B}_{\alpha,\delta}}
\def\cBud{\mathcal{B}_{1,\delta}}
\def\sB{\mathscr{B}}
\def\sC{\mathscr{C}}
\def\cE{\mathcal{E}}
\def\Fa{F_{\alpha}}
\def\sGa{\mathscr{G}_{\alpha}}
\def\cId{\mathcal{I}_{\delta}}

\def\cL{\mathcal{L}}
\def\cLR{\mathcal{L}_{R}^{+}}
\def\cLS{\mathcal{L}_{S}^{+}}
\def\sL{\mathscr{L}}
\def\sLu{\mathscr{L}_1}
\def\cLa{\mathscr{L}_{\alpha}}
\def\cM{\mathcal{M}}
\def\Na{\nu_{\alpha}}

\def\Oa{\Omega_{\alpha}}
\def\Od{\Omega_{\delta}}
\def\Plx{\Pi_{\Lambda,\xi}}
\def\Pla{\Pi_{\Lambda,a}}

\def\PLaz{\Pi_{\cLa,0}}
\def\PLzz{\Pi_{\sL_1,0}}
\def\PLza{\Pi_{\sL_1,-a_1}}
\def\PLaa{\Pi_{\cLa,-a_1}}
\def\cQa{\mathcal{Q}_{\alpha}}
\def\cQaS{\mathcal{Q}_{\alpha,S}}
\def\cQaR{\mathcal{Q}_{\alpha,R}}
\def\cQu{\mathcal{Q}_{1}}

\def\cQuR{\mathcal{Q}_{1,R}}

\def\cQe{\mathcal{Q}_e}
\def\Rl{R_{\Lambda}}
\def\cR{\mathcal{R}}
\def\cRl0{\cR_{\cL_0}}

\def\Sl{S_{\Lambda}}
\def\Slx{S_{\Lambda,\xi}}
\def\SLa{S_{\cLa}}
\def\cTu{\mathcal{T}_{1}}
\def\cTR{\mathcal{T}_{1,R}}
\def\cTS{\mathcal{T}_{1,S}}
\def\cTa{\mathcal{T}_{\alpha}}
\def\cTaR{\mathcal{T}_{\alpha,R}}
\def\cTaS{\mathcal{T}_{\alpha,S}}
\def\cTRR{\mathcal{T}_{R}}

\def\vv{\left\langle v\right\rangle}

\def\Td{\Theta_{\delta}}

\def\hu{\hat{u}}
\def\quand{\quad\text{and}\quad}

\newcommand{\supp}{\mathop{\mathrm{supp}}}
\newcommand{\vertiii}[1]{{\left\vert\kern-0.25ex\left\vert\kern-0.25ex\left\vert #1 
    \right\vert\kern-0.25ex\right\vert\kern-0.25ex\right\vert}}

\begin{document}

\title{Inelastic Boltzmann equation driven by a particle thermal bath}

\author{Rafael Sanabria}
\address{PUC--R. Marqu\^{e}s de S\~{a}o Vicente 225, Gavea. Rio de Janeiro, Brazil.}
\email{rafachesky@mat.puc-rio.br}



\date{\today}



\begin{abstract}
We consider the spatially inhomogeneous Boltzmann equation for inelastic hard-spheres, with constant restitution coefficient $\alpha\in(0,1)$, under the thermalization induced by a host medium with a fixed Maxwellian distribution and any fixed $e\in(0,1]$. When the restitution coefficient $\alpha$ is close to 1 we prove existence and uniqueness of global solutions considering the close-to-equilibrium regime. We also study the long-time behaviour of these solutions and prove a convergence to equilibrium with an exponential rate. 
\end{abstract}


\maketitle

\tableofcontents


\section{Introduction}

\subsection{Driven granular gases.}

Dilute granular flows are commonly modelled the by Boltzmann equation for inelastic hard-spheres interacting through binary collisions \cite{BrilliantovPoschel}. Due to dissipative collisions, energy continuously decreases in time which implies that, in absence of energy supply, the corresponding dissipative Boltzmann equation admits only trivial equilibria. This is no longer the case if the spheres are forced to interact with a forcing term, in which case the energy supply may lead to a non-trivial steady state. For such a driven system we consider hard spheres particles described by their distribution density $f=f(t,x,v)\geq 0$, $x\in\T^2$, $t>0$ satisfying 
\begin{equation}\label{eqn.BoltzmannEqn}
 \partial_t f+v\cdot\nabla_x f=\cQa(f,f)+\cL(f).
\end{equation}
where $\cQa(f,f)$ is the inelastic quadratic Boltzmann collision operator (see \ref{sec.Kinetic} for a precise definition), and $\cL(f)$ models the forcing term. The parameter $\alpha\in (0,1)$ is the so called ``restitution coeficient'' that characterized the inelasticity of the binary collisions. It quantifies the loss of relative normal velocity of a pair of colliding particles after the collision, with respect to the impact velocity (see \cite[Chapter~ 2]{BrilliantovPoschel}).The purely elastic case is recovered when $\alpha=1$.\\

In the literature there exist several possible physically meaningful choices for the forcing term $\cL$, in order to avoid the cooling of the granular gas. The first one is the pure diffusion thermal bath, studied in \cite{GambaPanferovVillani, MischlerMouhotSteadyState, Tristani}, for which
\[
 \cL_1(f)=\mu\Delta_v f,
\]
where $\mu>0$ is a given parameter and $\Delta_v$ the Laplacian in the velocity variable. Other fundamental examples of forcing terms are the thermal bath with linear friction 
\[
 \cL_2(f)=\mu\Delta_v f+\lambda \div(vf),
\]
where $\mu$ and $\lambda$ are positive constants and $\div$ is the divergence operator with respect to the velocity variable. Also, we have to mention the fundamental example of anti-drift forcing term which is related to the existence of self-similar solution to the inelastic Boltzmann equation:
\[
 \cL_3(f)=-\lambda \div(vf), \quad \lambda>0.
\] 
This forcing term has been treated in \cite{MischlerMouhotSelfSimilarity,MischlerMouhotSelfSimilarity2} for hard spheres.

\subsection{Description of the problem and main results.}

In this paper we consider a situation in which the system of inelastic hard spheres is immersed into a
thermal bath of particles, so that the forcing term $\cL$ is given by a linear scattering operator describing inelastic collisions with the background medium. More precisely, the forcing term $\cL$ is given by a linear Boltzmann collision operator of the form
\begin{equation}\label{eqn.DefL}
 \cL(f):=\cQe(f,\cM_0),
\end{equation}
where $\cQe(\cdot,\cdot)$ is the Boltzmann collision operator associated to the fixed restitution coefficient $e\in(0,1]$, and $\cM_0$ stands for the distribution function of the host fluid which we
assume to be a given Maxwellian with unit mass, bulk velocity $u_0$ and temperature $\theta_0 > 0$:
\begin{equation}\label{eqn.Maxwellian}
 \cM_0(v)=\left(\frac{1}{2\pi\theta_0}\right)^{\frac{3}{2}}\exp\left(-\frac{(v-u_0)^2}{2\theta_0}\right),\quad v\in\R^3.
\end{equation}

An important feature of the collision operators $\cQa(f,f)$ and $\cL(f)$ is that they both preserve mass. That is
\[
 \int_{\R^3}\cQa(f,f)dv= \int_{\R^3}\cL(f) dv=0.
\]
However, only the operator $\cQa$ preserves momentum. Neither the momentum nor the energy are conserved by the $\cL$ operator.\\

The existence of smooth stationary solutions $\Fa$ for the inelastic Boltzmann equation under the thermalization given by the forcing term $\cL$ has been proved in \cite{BisiCarrilloLods}. Uniqueness of the steady state is proven in \cite{BisiCanizoLodsUniqueness} for a smaller range of parameters $\alpha$. Our main result is the proof of existence of solutions for the non-linear problem (\ref{eqn.BoltzmannEqn}) near the equilibrium $\Fa$, as well as stability and relaxation to equilibrium for these solutions (a precise statement is given in Section \ref{sec.Nonlinear}): 

\begin{theorem}
Consider the functional spaces $\cE = W_x^{s,1} L^1_v(e^{b\vv^{\beta}})$ and $\cE_1 = W_x^{s,1} L^1_v(\vv e^{b\vv^{\beta}})$ where $b>0$, $\beta\in(0,1)$ and $s>6$. For $\alpha$ close to 1, for any $e\in(0,1]$, and for an initial datum $f\in\cE_1$ close enough to the equilibrium $\Fa$, there exists a unique global solution $f\in L^{\infty}_t(\cE)\cap L^1_t(\cE_1)$ to  \eqref{eqn.BoltzmannEqn} which furthermore satisfies for all $t\geq 0$,
\[
 \|f_t-\Fa\|_{\cE_0}\leq C e^{-at}\|f_{in}-\Fa\|_{\cE_0},
\]
for some constructive constants $C$ and $a$.
\end{theorem}

\subsection{Strategy of the proof and organization of the paper.}

In the inhomogeneous elastic case the Cauchy problem is usually handled by the theory  of perturbative solutions. This is based on the study of the linearized associated operator. However, this strategy was not available in the inelastic case, due to the absence of precise spectral study of the linearized problem. Another well-known theory in the elastic case is the one of DiPerna-Lions renormalized solutions
\cite{DipernaLions} which is no longer available in the inelastic case due to the lack of entropy estimates for the inelastic Boltzmann equation.\\

The recent work of Gualdani, Mischler, Mouhot \cite{GualdaniMischlerMouhot} presented a new technique to the spectral study of the elastic inhomogeneous regime. They presented an abstract method for deriving decay estimates on the resolvents and semigroups of non-symmetric operators in Banach spaces in terms of estimates in another smaller reference Banach space. As a consequence, they obtained the first constructive proof of exponential decay, with sharp rate, towards global equilibrium for the full nonlinear Boltzmann equation for hard spheres, conditionally to some smoothness and (polynomial) moment estimates. Furthermore, their strategy inspired several works in the kinetic theory of granular gases like \cite{AlonsoBaglandLods,AlonsoGambaTaskovic,BisiCanizoLodsUniqueness,BisiCanizoLodsEntropy,CanizoLods}. Using Gualdani et. al. approach, Tristani \cite{Tristani} was able to develop a perturbative argument around the elastic case in the same line as the one developed by Mischler, Mouhot \cite{MischlerMouhotSelfSimilarity,MischlerMouhotSteadyState}.\\

The strategy in this paper consists in combining the main ideas adopted in \cite{Tristani} with the arguments given by \cite{BisiCanizoLodsUniqueness} and \cite{CanizoLods}. To develop a Cauchy theory for the equation \eqref{eqn.BoltzmannEqn}, we first study the linearized problem around the equilibrium. Thus, we linearize our equation with the ansatz $f=\Fa+h$. Let us denote by $\cLa$ the linearized operator obtained by this ansatz. That is 
\[
 \cLa(h)=\cQa(\Fa,h)+\cQa(h,\Fa)+\cL(h)-v\cdot\nabla_x h.\\
\]

The study of the elastic case consists in deducing the spectral properties in $L^1$ from the well-known spectral analysis in $L^2$. This can be done thanks to a suitable splitting of the linearized operator as $\sL_1=\cA+\cB$, where $\cA$ is bounded and $\cB$ is ``dissipative'' operator, which are defined through an appropriate mollification-truncation process. This process is done in the same line as \cite{GualdaniMischlerMouhot} but incorporating the ideas of \cite{BisiCanizoLodsUniqueness} for the splitting of the forcing term $\cL$. We conclude that the spectrum of the linearized elastic operator is well localized. That means, it has a spectral gap in a large class of Sobolev spaces.\\

A crucial point in our approach is that it strongly relies on the understanding of the elastic problem corresponding to $\alpha=1$. Due to the properties of the equilibrium $\Fa$ presented in \cite{BisiCanizoLodsUniqueness}, we are able to prove that 
\begin{equation}\label{eqn.DifLaL1}
 \cLa-\sL_1=O(1-\alpha),
\end{equation}
for a suitable norm operator. Thus, one deduces the spectral properties of $\cLa$ from those of the elastic operator by a perturbation argument valid for $\alpha$ close enough to $1$. Notice that we only restrict the range of $\alpha$, $e$ is independent of $\alpha$ and can be taken in $(0,1]$. \\

Moreover, as in the case of the linearized operator, we obtain a splitting  $\cLa=\cA_{\alpha}+\cB_{\alpha}$, where $\cB_{\alpha}$ enjoys some dissipative properties and $\cA_{\alpha}$ some regularity properties. Combining this with the well localization of the spectrum of $\sL_1$ and \eqref{eqn.DifLaL1} we are able to deduce some properties of the spectrum of $\cLa$ valid for $\alpha$ close to 1. Moreover, we are able to obtain an estimate on the semigroup thanks to a spectral mapping theorem.\\

Regarding the nonlinear problem, one can build a solution by the use of an iterative scheme whose convergence is ensured due to a priori estimates coming from estimates of the semigroup of the linearized operator. A key element is an estimate for the bilinear collision operator established by Tristani \cite{Tristani}. For a sufficiently close to the equilibrium initial datum, the nonlinear part of the equation is small with respect to the linear part which dictates the dynamic. Therefore, we can recover an exponential decay to equilibrium for the nonlinear problem.\\

The organization of this paper is as follows. After recalling the precise definitions of the Boltzmann operator $\cQa$ and the forcing therm $\cL$ in Section \ref{sec.Prelim}, we proceed to define the function spaces as well as some spectral notations and definitions. The main known results are presented in Section \ref{sec.StadeStateKnownResults}. In Section \ref{sec.PrelimSS} we linearize the inelastic Boltzmann equation and present some important properties regarding the stady states. We begin by introducing the splitting of our forcing term $\cL$ as the sum of a regularizing part and a dissipative part in Section \ref{sec.ForcingSplit}. Moreover, we  prove existence of a spectral gap for the elastic linearized operator as well as decay rate for the linearized semigroup in Section \ref{sec.ElastLinOp}.\\

In Section \ref{sec.LinOp} we show that the inelastic linearized operator is a small perturbation of the elastic one. We also make a fine careful of spectrum close to 0, which allows us to prove existence of a spectral gap. Furthermore, we obtain a property of semigroup decay in $W_x^{s,1} W_v^{2,1}(\vv e^{b\vv^{\beta}})$ with $b>0$ and $\beta\in(0,1)$. Finally, we go back to the nonlinear Boltzmann equation in Section \ref{sec.Nonlinear} and prove our main result.\\


\section{Preliminaries}\label{sec.Prelim}

\subsection{Kinetic model}\label{sec.Kinetic}

We assume the granular particles to be perfectly smooth hard spheres of mass $m=1$ performing inelastic collisions. In the model at stake, the inelasticity is characterized by the so-called normal restitution coefficient $\alpha\in(0,1)$. The restitution coefficient quantifies the loss of relative normal velocity of a pair of colliding particles after the collision with respect to the impact velocity (see \cite[Chapter~ 2]{BrilliantovPoschel}). More precisely, if $v$ and $v_*$ (resp. $v'$ and $v'_*$) denote the velocities of a pair of particles before (resp. after) collision, we have the following equalities
\begin{equation}\label{eqn.PrePost}
\left\lbrace\begin{matrix}
  v+v_*=v'+v'_*,\\
  u'\cdot n=-\alpha(u\cdot n),
\end{matrix}\right.
\end{equation}
where $u=v-v_*$, $u'=v'-v'_*$ and $n\in\S^2$ stands for the unit vector that points from the $v$-particle center to the $v_*$-particle center at the moment of impact. The velocities after collision are then given by\\
\begin{equation*}
v'=v-\frac{1+\alpha}{2}(u\cdot n)n,\quad v'_*=v_*+\frac{1+\alpha}{2}(u\cdot n)n.
\end{equation*}
In particular, the rate of kinetic energy dissipation is
\begin{equation}\label{eqn.EnergyDissipation}
 |v'|^2+|v'_*|^2-|v|^2-|v_*|^2\leq -\frac{1-\alpha^2}{4}(u\cdot n)^2\leq 0.\\
\end{equation}

Another parametrization that shall be more convenient in the sequel is the following (see \cite{Tristani, BisiCarrilloLods}). If $v$ and $v_*$ are the velocities of two particles with $v\neq v_*$, we set $\hu=u/|u|$. Then, performing in (\ref{eqn.PrePost}) the change of unknown $\sigma = \hu -2 (\hu\cdot n)n\in\S^2$, it gives an alternative $q$ parametrization of the unit sphere $\S^2$. The impact velocity then writes $|u \cdot n|=|u|\sqrt{\frac{1-\hu\sigma}{2}}$. And the post-collisional velocities $v'$ and $v'_*$ are given by
\begin{equation*}
v'=v-\frac{1+\alpha}{2}\cdot\frac{u-|u|\sigma}{2},\quad v'_*=v_*+\frac{1+\alpha}{2}\cdot\frac{u-|u|\sigma}{2}.
\end{equation*}

Given a constant restitution coefficient $\alpha\in(0,1)$, one defines the weak form of the bilinear Boltzman operator $\cQa$ for inelastic interactions and hard spheres by its action on test functions $\phi(v)$ (see for example \cite[Section~ 2]{CanizoLods}),
\begin{align*}
\int_{\R^3}&\cQa(f,g)\phi(v)dv=\int_{\R^3\times\R^3\times\S^2}g(v_*)f(v)[\phi(v')-\phi(v)]|v-v_*|d\sigma dv_* dv\\\nonumber
&=\frac{1}{2}\int_{\R^3\times\R^3\times\S^2}g(v_*)f(v)[\phi(v'_*)+\phi(v')-\phi(v_*)-\phi(v)]|v-v_*|d\sigma dv_* dv.\\\nonumber
\end{align*}

\subsection{Function spaces}

Let us introduce the notations we shall use in the sequel (see for instance \cite{BisiCarrilloLods, Tristani}). Throughout the paper we shall use
the notation $\vv=\sqrt{1+|v|^2}$. For any $p,q\geq1$ and any weight $m>0$ on $\R^3$ we define the weighted Lebesgue space

\begin{equation*}
 L^p_xL^q_{v}(\vv m):=\{f:\R^3\to\R\text{ measurable}: \|f\|_{L^p_{x}L^q_{v}(\vv m)}<+\infty\},\\
\end{equation*}

where the norm $\|f\|_{L^p_xL^q_v(\vv m)}$ is defined by

\begin{equation*}
 \|f\|_{L^p_{x}L^q_v(\vv m)}:=\|\|f(\cdot,v)\|_{L^p_x}\vv m(v)\|_{L^q_v}.\\
\end{equation*}

The weighted Sobolev space $W^{s,p}_{x}W^{\sigma,q}_v$ for any $p,q\geq1$ and $\sigma,s\in\N$ is defined by the norm
\begin{align*}
 \|f&\|_{W^{s,p}_{x}W^{\sigma,q}_v (\vv m)}:=\\ 
 &\sum_{0\leq s'\leq s,0\leq\sigma'\leq\sigma,s'+\sigma'\leq\max(s,\sigma)}\|\|\nabla^{s'}_x \nabla^{\sigma'}_v f(\cdot,v)\|_{L^p_x}\vv m(v)\|_{L^q_v}.
\end{align*}


Moreover, if $s>1$ and is not an integer then we write $s=t+r$, where $t\in\Z$ and $r\in(0,1)$. In this case, the space $W^{s,p}_{x}W^{\sigma,q}_v(\vv m)$ consist of those equivalence classes of functions $f\in W^{t,p}_{x}W^{\sigma,q}_v(\vv m)$ whose distributional derivatives $\nabla_x^{t}f$ belong to $W^{r,p}_{x}W^{\sigma,q}_v(\vv m)$. This last space is determined by the norm

\begin{align*}
 \|&f\|_{W^{r,p}_{x}W^{\sigma,q}_v(\vv m)}:=\|f\|_{L^p_{x}W^{\sigma,q}_v(\vv m)}\\
 &+\sum_{0\leq\sigma'\leq\sigma,t+\sigma'\leq\max(t,\sigma)}\left(\int_{\R^3}\left(\int_{\R^3}\frac{|\nabla_v^{\sigma'}f(x,v)-\nabla_v^{\sigma'}f(y,v)|}{|x-y|^{4+rp}}^pdxdy\right)^{q/p}\vv m dv\right)^{1/q}.\\
\end{align*}

\subsection{Spectral notations}

Given a real number $a\in\R$ let us define
\[
 \Delta_a:=\{z\in\C:\Re z>a\}.
\]

For some given Banach spaces $(E, \|\cdot\|_E)$ and $(F, \|\cdot\|_F )$, we denote the space of bounded linear operators from $E$ to $F$ by $\sB(E,F)$ and we denote by $\|\cdot\|_{\sB(E,E)}$ or $\|\cdot\|:E\to E$ the associated operator norm. We write $\sB(E) = \sB(E,E)$ when $F = E$. Moreover, we denote by $C(E,F)$ the space of closed unbounded linear operators from $E$ to $F$ with dense domain, and $C(E)=C(E,E)$ in the case $E=F$.\\

For a Banach space $X$ and $\Lambda\in C(X)$ its associated semigroup is denoted by $S_{\Lambda}(t)$, for $t\geq 0$, when it exists. Also denote by $D(\Lambda)$ its domain, by $N(\Lambda)$ its null space and by $R(\Lambda)$ its range. Let us introduce the D$(\Lambda)$-norm defined as 
\[
 \|f\|_{D(\Lambda)} = \|f\|_X + \|\Lambda f\|_X\text{ for }f \in D(\Lambda).
\]
More generally, for every $k \in\N$, we define
\[
 \|f\|_{D(\Lambda^k)} = \sum_{j=0}^k\|\Lambda^j f\|_X,\quad f\in D(\Lambda^k).
\]

Its spectrum is denoted by $\Sigma(\Lambda)$, and the resolvent set $\rho(\Lambda):=\C \setminus (\Sigma(\Lambda))$. So for any $z\in\rho(\Lambda)$ the operator $\Lambda-z$ is invertible and the resolvent operator
\[
 \Rl(z):=(\Lambda-z)^{-1}
\]
is well-defined, belongs to $\sB(X)$ and has range equal to $D(\Lambda)$. A number $\xi\in\Sigma(\Lambda)$ is said to be an eigenvalue if $N(\Lambda-\zeta)=\{0\}$. Moreover, an eigenvalue $\xi\in\Sigma(\Lambda)$ is said to be isolated if there exists $r > 0$ such that
\[
 \Sigma(\Lambda)\cap \{|z\in\C:|z-\xi|<r\}=\{\xi\}.
\]

If $\xi$ is an isolated eigenvalue we define the associated spectral projector 
\[
 \Plx:=-\frac{1}{2\pi i}\int_{|z-\xi|=r'}R_{\Lambda}(z)dz\in\sB(X),
\]
which is independent of $0<r'<r$ since $z\mapsto \Rl(z)$ is holomorphic. It is well-known that $\Plx^2=\Plx$ so it is a projector and the ``associated projected semigroup'' is
\[
 \Slx(t):=-\frac{1}{2\pi i}\int_{|z-\xi|=r'}e^{tz}\Rl(z)dz,\quad t>0;
\]  
which satisfies that for all $t>0$ 
\[
 \Slx(t)=\Plx\Slx(t)=\Slx\Plx(t).
\]
When the ``algebraic eigenspace'' $R(\Plx)$ is finite dimensional we say that $\xi$ is a discrete eigenvalue, written as $\xi\in\Sigma_d(\Lambda)$. For more about these results we refer the reader to \cite[Chapter~III-6]{Kato}.\\

Finally for any $a\in\R$ such that $\Sigma(\Lambda)\cap\Delta_a=\{\xi_1,...,\xi_k\}$ where $\xi_1,...,\xi_k$ are distinct discrete eigenvalues, we define without ambiguity
\[
 \Pla=\Pi_{\Lambda,\xi_1}+\cdots+\Pi_{\Lambda,\xi_k}.
\]

If one considers some Banach spaces $X_1$, $X_2$, $X_3$, for two given functions $S_1\in L^1(\R_+,\sB(X_1,X_2))$ and $S_2\in L^1(\R_+,\sB(X_2,X_3))$, the convolution 
\[
S_1\ast S_2\in L^1(\R_+,\sB(X_1,X_3)),
\]
is defined for all $t\geq0$ as 
\[
  (S_1\ast S_2)(t):=\int_0^tS_2(s)S_1(t-s)ds.
\]

When $S_1 = S_2$ and $X_1 = X_2 = X_3$, $S^{(\ast l)}$ is defined recursively by $S^{(\ast 1)}=S$ and for any $l\geq 2$, $S^{(\ast l)} = S \ast S^{(\ast(l-1))}$.

\subsection{Hypodissipative operators}

Let us introduce the notion of \textit{hypodissipative} operators. Consider a Banach space $(X, \|\cdot\|_X)$ and some operator $\Lambda\in C(X)$. $(\Lambda-a)$ is said to be hypodissipative on $X$ if there exists some norm $|||\cdot|||_X$ on $X$ equivalent to the initial norm $\|\cdot\|_X$ such that for every $f\in D(\Lambda)$ there exist $\phi\in F(f)$ such that
\[
 \Re\left\langle \phi,(\Lambda-a)f\right\rangle\leq 0,
\]
where $\left\langle\cdot,\cdot\right\rangle$ is the duality bracket for the duality in $X$ and $X^{*}$ and, $F(f)\subset X^{*}$ is the dual set of $f$ defined by
\begin{equation}\label{eqn.DefFDual}
 F(f):=\{\phi\in X^{*}: \left\langle\phi,f\right\rangle=|||f|||_{X^{*}}=|||f|||_X\}.
\end{equation}

The following theorem is a non standard formulation of the classical Hille-Yosida theorem on $m$-dissipative operators and semigroups. It summarizes the link between PDE's, the semigroup theory and spectral analysis. For the proof of this result we refer the reader to \cite[Chapter~ 1]{Pazy},  and for more about hypodissipative operators see \cite[Section~ 2.3]{GualdaniMischlerMouhot} and \cite[Section~2.1]{MischlerScher}.\\

\begin{theorem}\label{thm.HypodEquivalence}
Consider $X$ a Banach space and $\Lambda$ the generator of a $C^0$-semigroup $\Sl$. We denote by $\Rl$ its resolvent. For given constants $a\in \R$ and, $M>0$ the following assertions are equivalent:
\begin{enumerate}
\item $\Lambda -a$ is hypodissipative;
\item the semigroup satisfies the growth estimate for every $t\geq 0$
\[
 \|\Sl(t)\|_{\sB(X)}\leq Me^{at};
\]
\item $\Sigma(\Lambda)\cap \Delta_a=\emptyset$ and for all $z\in\Delta_a$
\[
 \|\Rl(z)^n\|\leq\frac{M}{(\Re z -a)^n};
\]
\item $\Sigma(\Lambda)\cap (a,\infty)=\emptyset$ and there exist some norm $|||\cdot|||$ on $X$ equivalent to the norm $\|\cdot\|$ such that for all $f\in X$
\[
 \|f\|\leq |||f|||\leq M\|f\|,
\]
and such that for every $\lambda>a$ and every $f\in D(\Lambda)$
\[
 |||(\Lambda-\lambda)f|||\geq (\lambda-a)|||f|||.
\]
\end{enumerate}
\end{theorem}


\section{Main known results}\label{sec.StadeStateKnownResults}

The existence of smooth stationary solutions for the inelastic Boltzmann equation under the thermalization given by (\ref{eqn.DefL}) has already been proved by Bisi, Carrillo, Lods in \cite[Theorem~ 5.1]{BisiCarrilloLods} for any choice of restitution coefficient $\alpha$. Moreover, the uniqueness of the solution was  obtained by Bisi, Ca\~{n}izo, Lods in \cite[Theorem~1.1]{BisiCanizoLodsUniqueness}. These results can be summarized as follows:

\begin{theorem}\label{thm.ExistenceUniqSS}
For any $\rho>0$ and $\alpha\in(0,1]$, there exists a steady solution $\Fa\in L^1_v(\vv^2)$, $F_{\alpha}(v)\geq 0$ to the problem 
\begin{equation}\label{eqn.StationaryEqn}
 \cQa(\Fa,\Fa)+\cL(\Fa)=0,
\end{equation}
with $\int_{\R^3}\Fa(v)dv=\rho$. Moreover, there exists $\alpha_0\in(0, 1]$ such that such a solution is unique for $\alpha\in(\alpha_0, 1]$. This (unique) steady state is radially symmetric and belongs to $C^{\infty}(\R^3)$.
\end{theorem}

Let us denote by $\sGa$ the set of functions $\Fa$ solutions of (\ref{eqn.StationaryEqn}) with mass 1, that satisfies $\int_{\R^3}\Fa(v)dv=1$. We recall a quantitative estimate on the distance between $\Fa$ and the Maxwellian $\cM$:
\begin{equation}\label{eqn.DefM}
\cM(v)=\left(\frac{1}{2\pi\theta^{\sharp}}\right)^{3/2}\exp\left\lbrace-\frac{(v-u_0)}{2\theta^{\sharp}}\right\rbrace,
\end{equation} 
with $\theta^{\sharp}=\frac{1+e}{3-e}\theta_0$ where $\theta_0$ is as in (\ref{eqn.Maxwellian}). This Maxwellian is the unique solution of (\ref{eqn.StationaryEqn}) in the elastic case, i.e. when $\alpha=1$, (see \cite[Theorems~2.3 and ~5.5]{BisiCanizoLodsUniqueness}).

\begin{theorem}\label{thm.ConvergenceFaM}
There exist an explicit function $\eta_1(\alpha)$ such that $\lim_{\alpha\to 1}\eta_1(\alpha)=0$ and such that for any $\alpha_0\in(0,1]$
\[
 \sup_{\Fa\in\sGa}\|\Fa-\cM\|_{\mathcal{Y}}\leq\eta_1(\alpha),\quad\forall\alpha\in(\alpha_0,1].
\]
Here, $\mathcal{Y}=L^1_v(\vv e^{a|v|})$ with $a>0$.
\end{theorem}

The weak form of the collision operator $\cQe$ suggests the natural splitting between gain and loss parts $\cL=\cL^+-\cL^-$. For the loss part notice that 
\begin{equation*}
\left\langle\cL^-(f,g),\psi\right\rangle=\int_{\R^3\times\R^3\times\S^2}g(v_*)f(v)\psi(v)|v-v_*|d\sigma dv_* dv=\left\langle fL(g),\psi\right\rangle,
\end{equation*}
where $\left\langle\cdot,\cdot\right\rangle$ represents the inner product in $L^2$ and $L$ is the convolution operator
\begin{equation}\label{eqn.DefColFreqNuE}
 L(g)(v)=4\pi(|\cdot|\ast g)(v).\\
\end{equation}

\begin{remark}
Notice that $L$ and $\cL^-=\cQe^-$ are independent of the restitution coefficient.
\end{remark}

Let us introduce the collision frequency $\nu_e:=L(\cM_0)$ and consider 
\begin{equation}\label{eqn.NuE0}
 \nu_e^0=\inf_{v\in\T^3}\nu_e(v)>0.
\end{equation}
 
It is easy to see that $\nu(v)\approx\vv$. In other words, there exist some constants $\nu_{e,0},\nu_{e,1}>0$ such that for every $v\in\R^3$
\begin{equation}\label{eqn.ColFreqIneq}
 0<\nu_{e,0}\leq \nu_{e,0}\vv\leq\nu_e(v)\leq \nu_{e,1}\vv.
\end{equation}

Consider the space $\mathscr{H}=L^2_v(\cM^{-1/2})$. Arlotti and Lods in \cite[Theorem~3.7]{ArlottiLods} performed the spectral analysis of $\cL$ in $\mathscr{H}$. Moreover, Mouhot, Lods and Toscani in \cite{LodsMouhotToscani} give some quantitative estimates of the spectral gap. These results can be summarized in the following: \\

\begin{theorem}\label{thm.SpectrumQe}
Consider $\left\langle \cdot,\cdot\right\rangle_{\mathscr{H}}$ the inner product in the Hilbert space $\mathscr{H}$. Then we have the following description of the spectrum of $\cL$:
\begin{enumerate}
\item $\cL^+$ is compact in $\mathscr{H}$.\\
\item The spectrum of $\cL$ as an operator in $\mathscr{H}$ consist of the spectrum of $-\cL^-$ and of, at most, eigenvalues of finite multiplicities. Precisely, we have
\[
 \Sigma(\cL)=\{\lambda\in\R:\lambda\leq-\nu_e^0\}\cup\{\lambda_n:n\in I\},
\]
where $I\subset\N$ and $(\lambda_n)_n$ is a decreasing sequence of real eigenvalues of $\cL$ with finite algebraic multiplicities: $\lambda_0=0>\lambda_1>\lambda_2\cdots>\lambda_n>\cdots$, which unique possible cluster point is $-\nu_e^0$, where $\nu_e^0$ was defined in \ref{eqn.NuE0}.\\
\item $\cL$ is a nonnegative self-adjoint operator and there exists $\mu_e>0$ (the spectral gap) such that
\[
 -\left\langle h,\cL(h)\right\rangle_{\mathscr{H}}\geq \mu_e\|h-\rho_h \cM\|_{\mathscr{H}},
\]
where $\rho_h=\int h dv$.\\
\item $0$ is a simple eigenvalue of $\cL$ with $N(\cL)=Span\{\cM\}$
\end{enumerate}
 
\end{theorem}


\section{Preliminaries on steady states}\label{sec.PrelimSS}


Let $\Fa$ be the steady state given by Theorem \ref{thm.ExistenceUniqSS}. We linearize our equation around the equilibrium $\Fa$ with the perturbation $f=\Fa+h$. That is, by substituting $f$ in (\ref{eqn.BoltzmannEqn}) we obtain
\begin{equation}\label{eqn.InitialValueProblem}
 \partial_t h=\cQa(h,h)+\cLa(h).
\end{equation}
where $\cLa(h)=\cQa(\Fa,h)+\cQa(h,\Fa)+\cL(h)-v\cdot\nabla_x h$.\\

If we consider only the linear part we obtain the first order linearized equation around the equilibrium $\Fa$
\begin{equation}\label{eqn.LinearEqn}
 \partial_t h=\cLa(h).\\
\end{equation}

Throughout the paper, we shall use the notation $\vv=\sqrt{1+|v|^2}$ and denote
\[
 m(v):=\exp\left(b\vv^{\beta}\right),
\]
with $b>0$ and $\beta\in(0,1)$. Let us state several lemmas on steady states $\Fa$ that will we needed several times in the future. First of all, we prove an estimate for the Sobolev norm

\begin{lemma}\label{lem.BoundSobolevNorm}
Let $k,q\in\N$. We denote $k'=8k+7(1+3/2)$. Then there exist $C>0$ such that
\[
 \|\Fa\|_{W^{k,1}_v(\vv^q m)}\leq C.
\]
\end{lemma}

\begin{proof}

First we recall that by \cite[Theorem~ 3.3]{BisiCanizoLodsUniqueness} there exist some constants $A > 0$ and $M > 0$ such that, for any $\alpha\in(0, 1]$ and any solution $\Fa$ to (\ref{eqn.StationaryEqn}) one has
\[
 \int_{\R^3}\Fa(v)e^{a|v|^2}dv<M.\\
\]

Using the inequalities of Cauchy and Bernoulli we get that, if $6\beta<A$ 
\[
\int_{\R^3}\Fa(v)m(v)dv\leq e^{(b^2+1)/2}M=:C_1,\\ 
\] 
and 
\[
\int_{\R^3}\Fa(v)m^{12}(v)dv\leq e^{6(b^2+1)}M=:C_2.\\
\]

Moreover, from \cite[Corollary~3.6]{BisiCanizoLodsUniqueness} we know that for any $k\in\N$, there exist $C_k>0$ such that $\|\Fa\|_{H^k_v}\leq C_k$. Thus, using Lemma \ref{lem.Interpolation} we get
\[
 \|\Fa\|_{W^{k,1}_v(\vv^q m)}\leq C C_{k'}^{1/8}C_1^{1/8}C_2^{3/4}=:C',
\]
that concludes our proof.
\end{proof}
 
Now we estimate the difference between $\Fa$ and the elastic equilibrium $\cM$, which is the Maxwellian given in (\ref{eqn.DefM}).

\begin{lemma}\label{lem.DifferenceFaM}
Let $k,q\in\N$. We denote $k'=8k+7(1+3/2)$. Then there exists a function $\eta(\alpha)$ such that for any $\alpha\in(\alpha_0,1]$
\[
 \|\Fa-\cM\|_{W^{k,1}_{v}(\vv^q m)}\leq \eta(\alpha),
\]
with $\eta(\alpha)\to 0$ when $\alpha\to 1$.
\end{lemma}
\begin{proof}
By Theorem \ref{thm.ConvergenceFaM} there exists an explicit function $\eta_1(\alpha)$ such that $\displaystyle\lim_{\alpha\to1}\eta_1(\alpha)=0$ and such that for any $\alpha_0\in(0,1]$ and for every $\alpha\in(\alpha_0,1]$
\[
 \|\Fa-\cM\|_{L^1_v(\vv m)}\leq \eta_1(\alpha).
\]
Since, $1\leq\vv$ for every $v\in\R^3$
\[
 \|\Fa-\cM\|_{L^1_v(m)}\leq \eta_1(\alpha).
\]
Then, using Lemma \ref{lem.Interpolation} and the proof of Lemma \ref{lem.BoundSobolevNorm} we get
\[
 \|\Fa-\cM\|_{W^{k,1}_{v}(\vv^q m)}\leq C(2C_{k'})^{1/8}(2C_2)^{1/8}\eta_1^{3/4}(\alpha)=:\eta(\alpha),
\]
which concludes our proof.
\end{proof}

\section{The forcing term and its splitting.}\label{sec.ForcingSplit}

Consider the following Banach spaces for $s\geq 0$
\begin{align*}
&E_1:=W_{x}^{s+2,1}W_{v}^{4,1}(\vv^2 m),\\
&E_0=E:=W_{x}^{s,1}W_{v}^{2,1}(\vv m),\\
&\cE:=W_{x}^{s,1}L_{v}^{1}(m).
\end{align*}

As mentioned before, it is well-known that $0$ is a simple eigenvalue associated to the eigenfunction $\cM$ for the operator $\cL$ and it admits a positive spectral gap in $\mathscr{H}=L^2_v(\cM^{-1/2})$, where $\cM$ is defined in \eqref{eqn.DefM} . The aim of this section is to show that the same is true for $\sL_1$ in the larger spaces $E_0,E_1$.\\

\subsection{Splitting}

For any $\delta\in(0,1)$ consider the bounded (by one) operator $\Td=\Td(v,v_*,\sigma)\in C^{\infty}$, which equals one in
\[
 \left\lbrace|v|\leq \delta^{-1}, 2\delta\leq |v-v_*|<\delta^{-1}\quand |\cos \theta|\leq 1-2\delta\right\rbrace,
\]
and whose support is included in
\[
 \left\lbrace|v|\leq 2\delta^{-1}, \delta\leq |v-v_*|<2\delta^{-1}\quand |\cos \theta|\leq 1-\delta\right\rbrace.
\]
The introduction of the parameter $\delta$ in this truncation function will allow us to prove the properties we are looking for. Namely, it lets us decompose the operator $\cL$ in the following way $\cL=\cAed+\cBed$, where $\cAed$ has some regularity and $\cBed$ is hypodissivative.\\

Recall that $\nu_e:=L(\cM_0)$ where $L$ is the convolution operator defined by \eqref{eqn.DefColFreqNuE}. Therefore, we can write $\cL$ in the following way:
\[
 \cL(h)(v)=\cLS(h)(v)+\cLR(h)(v)-\nu_e(v)h(v),
\]
where $\cLS$ is the truncated operator given by $\Td$ and $\cLR$ the corresponding reminder.  By this we mean that for any test function $\psi$:
\begin{equation*}
\left\langle\cLS(h),\psi\right\rangle=\int_{\R^3\times\R^3\times\S^2}\Td h(v)\cM_0(v_*)\psi(v)|v-v_*|d\sigma dv_* dv,
\end{equation*}
and,
\begin{equation*}
\left\langle\cLR(h),\psi\right\rangle=\int_{\R^3\times\R^3\times\S^2}(1-\Td) h(v)\cM_0(v_*)\psi(v)|v-v_*|d\sigma dv_* dv.
\end{equation*}

Hence we can give a decomposition for the forcing term $\cL(h)=\cAed(h)+\cBed(h)$ where $\cAed(h):=\cLS(h)$ and $\cBed(h)=\cLR(h)-\nu_e h$. \\

By the Carleman representation for the inelastic case (see \cite[Theorem~1.4]{ArlottiLods} and \cite[Proposition~ 1.5]{MischlerMouhotSelfSimilarity2}), we can write the truncated operator $\cAed$ as
\begin{equation}\label{eqn.CalermanAed}
 \cAed(h)(v)=\int_{\R^3}k_{e,\delta}(v,v_*)h(v_*)dv_*,
\end{equation}
for some smooth kernel $k_{e,\delta}\in C^{\infty}_c(\R^3\times\R^3)$. The smoothness of the kernel allows us to prove a the following regularity estimate:\\

\begin{lemma}\label{lem.BoundednessAed}
For any $s\geq 0$ and any $e\in(0,1]$, the operator $\cAed$ maps $L^1_v$ into $H_v^{s+1}$ functions with compact support, with explicit bounds (depending on $\delta$) on the $L^1_v\mapsto H_v^{s+1}$ norm, and on the size of the support. More precisely, there are two constants $C_{s,\delta}$ and $R_{\delta}$ such that for any $h\in L^1_v$ 
\[
K:= \supp \cAed h\subset B(0,R_{\delta}),\quand \|\cAed\|_{H^{s+1}_v(K)}\leq C_{s,\delta} \|h\|_{L^1_v}.
\]
In particular, we deduce that $\cAed$ is in $\sB(E_j )$ for $j =0, 1$ and $\cAed$ is in $B(\cE, E)$.
\end{lemma}

\begin{proof}
It is clear that the range of the operator $\cAed$ consists of compactly supported functions thanks to the truncation. Moreover, the bound on the size of the support is related to $\delta$.\\

On the other hand, the proof of the smoothing estimate follows from \cite[Proposition~2.4]{AlonsoLods} since $\cLS$ is the gain part of the collision operator associated to the mollified collision kernel $B=\Theta_{\delta}|v-v_*|$. Even though their original statement is for functions in $L^1_v(\vv^{2\eta+s+4})$ for any $\eta\geq 0$, as they mention in their proof, in the case of compact support of the collision kernel it can be proved for functions in $L^1_v(\vv^{\eta})$. Therefore, taking $\eta=0$ we obtain our result.\\
\end{proof}

Furthermore, notice that from \cite[Lemma~2.6]{Tristani} we have that the operator $\cL$ is bounded from $W_{x}^{s,1}W_{v}^{k,1}(\vv^{q+1} m)$ to $W_{x}^{s,1}W_{v}^{k,1}(\vv^{q} m)$ for any $q>0$. Hence, is bounded from $E_1$ to $E_{0}$. \\

\subsection{Hypodissipativity of $\cBed$}


\begin{lemma}\label{lem.HypodissBed}
Let us consider $k\geq 0$, $s\geq k$ and $q\geq 0$. Then, there exist $\delta\geq0$ and $a_0>0$ such that for any $e\in(0,1]$, the operator $\cBed+a_0$ is hypodissipative in $W_{x}^{s,1}W_{v}^{k,1}(\vv^{q} m)$.
\end{lemma}

\begin{proof}
We consider the case $W^{1,1}_{x,v}(\vv^q m)$. The higher-order cases are  treated in a similar way.\\

Consider a solution $h$ to the linear equation $\partial_t h=\cBed(h)$ given an initial datum $h_0$. The main idea of the proof is to construct a positive constant $a_0$ and a norm $\|\cdot\|_*$ equivalent to the norm on $W^{1,1}_{x,v}(\vv^q m)$ such that there exist $\psi$ in the dual space $\left(W^{1,1}_{x,v}(\vv^q m)\right)^*$ with $\psi\in F(h)$, where $F(h)$ is defined in \eqref{eqn.DefFDual}, and
\begin{equation}\label{eqn.NormDerivative}
 \Re\left\langle \psi,\cBed h\right\rangle\leq -a_0\|h\|_*.\\
\end{equation}

We have divided the proof into four steps. The first one deals with the hypodissipativity of $\cBed$ in $L^1_xL^1_v(\vv^{q+1}m)$, while the second and third  deal with the $x$ and $v$-derivatives respectively. In the last step we construct the $\|\cdot\|_*$ norm and prove that it satisfies \eqref{eqn.NormDerivative}.\\

\textbf{Step 1:} Notice that, for $k=0$, the hypo–dissipativity of $\cBed$ simply reads 
\[
 \int_{\R^3\times\T^3}\cBed(h)\text{sign}(h)\vv^q m(v)dxdv\leq -a'_0\|h\|_{L^1_xL^1_v(\vv^{q+1} m)},
\]
for some positive constant $a'_0$. Which means that, for $k=0$, $\cBed$ is actually dissipative.\\

Let us recall that $\cBed(h)=\cLR(h)-\nu_e h$. Consider $\Od$ the set where $\Td$ equals $1$ that is
\[
 \Od:=\left\lbrace|v|\leq \delta^{-1}, 2\delta\leq |v-v_*|<\delta^{-1}\quand |\cos \theta|\leq 1-2\delta\right\rbrace.
\]
Our proof starts with the observation that $\Od^c=\Od^1\cup\Od^2\cup\Od^3\cup\Od^4$ where
\begin{align*}
&\Od^1= \left\lbrace |v|> \delta^{-1}\right\rbrace,\\
&\Od^2= \left\lbrace 2\delta> |v-v_*|\right\rbrace,\\
&\Od^3= \left\lbrace |v-v_*|\geq\delta^{-1}\right\rbrace,\\
&\Od^4= \left\lbrace |\cos \theta|> 1-2\delta\right\rbrace.
\end{align*}

Hence, we have that $1-\Td\leq \mathbbm{1}_{\Od^c}$. Thus, using the weak form of the collision operator for $\psi(v)=\text{ sign}(h)\vv^q m(v)$
\begin{align*}
\int_{\R^3}\cLR(h)\psi(v)dv&=\int_{\R^3\times\R^3}(1-\Td)h(v)\cM_0(v_*)|v-v_*|\psi(v')d\sigma dv dv_*\\
&\leq\sum_{j=1}^4\int_{\Od^j}|h(v)|\cM_0(v_*)|v-v_*|\left\langle v'\right\rangle^q m(v')d\sigma dv dv_*\\
&=:I_1+I_2+I_3+I_4.
\end{align*}

We first deal with the integral $I_2$. It is easy to see that for every $q\geq 0$, there exist a constant $C_q>0$ such that 
\begin{equation}\label{eqn.WeigthIneq}
 \left\langle v'\right\rangle^q m(v')\leq C_q \vv^q m(v) \left\langle v_*\right\rangle^q m(v_*).
\end{equation}
Thus, one has
\begin{align}\label{eqn.QeI2}\nonumber
I_2
   &\leq 2\delta C_q \int_{\Od^2}|h(v)|\cM_0(v_*)\vv^q m(v) \left\langle v_*\right\rangle^q m(v_*) d\sigma dv dv_*\\
   &\leq C_2\delta \|h\|_{L^1_v(\vv^{q}m)}, 
\end{align}
where $C_2=8\pi C_q\|\cM_0\|_{L^1_v(\vv^q m)}$.\\

We now turn to $I_4$. Notice that $|u|\leq 2\vv\left\langle v_*\right\rangle$ combining this with \eqref{eqn.WeigthIneq} there exist some positive constant $C'_4$ such that 
\begin{equation}\label{eqn.WeightIneqAbsolute}
 |v-v_*| \left\langle v'\right\rangle^q m(v')\leq C'_4 \vv^{q+1} m(v) \left\langle v_*\right\rangle^{q+1} m(v_*).
\end{equation}
Then using \eqref{eqn.WeightIneqAbsolute} we have 
\begin{align}\label{eqn.QeI4}
I_4
   &\leq C'_4 \int_{\Od^4} |h(v)|\cM_0(v_*)\vv^{q+1} m(v) \left\langle v_*\right\rangle^{q+1} m(v_*)d\sigma dv dv_*\\\nonumber
   &\leq C_4\Lambda(\delta) \|h\|_{L^1_v(\vv^{q+1}m)}, 
\end{align}
where $C_4=C'_4\|\cM_0\|_{L^1_v(\vv^q m)}$ and $\Lambda(\delta)$ is the measure of $\Od^4$ on the sphere $\S^2$ which goes to zero as $\delta$ goes to zero.\\

We can now proceed to analyze the integral $I_3$. According to the kinetic rate of energy dissipation \eqref{eqn.EnergyDissipation} and Lemma \ref{lem.ExistenceCbeta}
\[
 |v'|^{\beta}\leq\left(|v|^2+|v_*|^2-\frac{|v'_*|^2}{2}\right)^{\beta/2}\leq \left(|v|^{\beta}+|v_*|^{\beta}\right)-C'_{\beta}|v'_*|^{\beta},
\]
where in the last inequality we use the fact that $\beta\in(0,1)$. Therefore, we have that for some positive constant $C'_3$
\[
 \cM_0(v_*)m(v')\leq C'_3m(v)\exp\left\lbrace b|v_*|^{\beta}-\beta_0|v_*|^2-bC'_{\beta}|v'_*|^{\beta}\right\rbrace,
\]
here we are using the fact that $\cM_0(v)\leq C_e e^{-\beta_0|v|^2}$ for some positive constants $\beta_0$ and $C_e$. Thus, on account of Lemma \ref{lem.ExistenceCb}, for any $\gamma>0$,  we get
\[
 \cM_0(v_*)m(v')\leq C'_3 e^{C_{\gamma}}m(v)\exp\left\lbrace-\gamma|v_*|^{\beta}-bC'_{\beta}|v'_*|^{\beta}\right\rbrace.
\]
Take $\gamma=2bC'_{\beta}$, and since $((1+e)|u\cdot n|)/2\leq |v_*|+|v'_*|$, where $u=v-v_*$, we can conclude that
\begin{align*}
 \exp\left\lbrace-\gamma|v_*|^{\beta}-C'_{\beta}|v'_*|^{\beta}\right\rbrace&\leq\exp\left\lbrace-bC'_{\beta}|v_*|^{\beta}-bC'_{\beta}(|v_*|+|v'_*|)^{\beta}\right\rbrace\\
  &\leq \exp\left\lbrace-bC'_{\beta}|v_*|^{\beta}-\frac{bC'_{\beta}}{2^{\beta}}|u\cdot n|^{\beta}\right\rbrace.
\end{align*}
Hence, since $\left\langle v'\right\rangle^{q}\leq C_q \vv^{q}\left\langle v_*\right\rangle^{q}$, we have that for $C_{3,\beta}=4C'_3C_q\exp(2bC'_{\beta})$
\begin{equation}\label{eqn.QeI3A}
I_3  \leq C_{3,\beta}\int \mathbbm{1}_{\{|u|>\delta^{-1}\}}|h(v)|\vv^{q}\left\langle v_*\right\rangle^{q}m(v)\exp\left\lbrace-bC'_{\beta}|v_*|^{\beta}\right\rbrace|u|Jdv_*dv,
\end{equation}
where $J$ is the integral $ J=\int_{\S}\exp\left\lbrace-(bC'_{\beta}/2^{\beta})|u\cdot n|^{\beta}\right\rbrace dn$. Thus, recalling that $\hu=u/|u|$ and $(\hu\cdot n)=\cos \theta$ we get that
\[
 J = 2\int_{0}^{\pi/2}\exp\left\lbrace-(C'_{\beta}/2^{\beta})|u|^{\beta}|\cos \theta|^{\beta}\right\rbrace \sin \theta dn.\\
\]
Take $C_J=bC'_{\beta}/2^{\beta}$. Now consider the change of variables $z=\cos\theta$ and then $w=|u|z$ which transform the integral $J$ into
\[
 J=2\int_{0}^{1}e^{-C_J|u|^{\beta}z^{\beta}} dz=2|u|^{-1}\int_{0}^{|u|}e^{-C_J w^{\beta}} dw\leq C'_J|u|^{-1}.\\
\]
It is easy to see that $\mathbbm{1}_{\{|u|>\delta^{-1}\}}\leq \delta|u|$. Substituting the above inequality into the integral in \eqref{eqn.QeI3A} and using \eqref{eqn.WeightIneqAbsolute} we obtain
\begin{align}\label{eqn.QeI3}
\nonumber
 I_3 & \leq C_{3,\beta}C'_J\int \mathbbm{1}_{\{|u|>\delta^{-1}\}}|h(v)|\vv^{q}\left\langle v_*\right\rangle^{q}m(v)\exp\left\lbrace-C_{\beta}|v_*|^{\beta}\right\rbrace dv_*dv\\\nonumber
     & \leq C''_{3}\delta\int \exp\left\lbrace-C_{\beta}|v_*|^{\beta}\right\rbrace \left\langle v_*\right\rangle^{q+1} dv_*\int |h(v)|\vv^{q+1}m(v)dv\\
     &\leq C_3\delta\|h\|_{L^1(\vv^{q+1}m)}.
\end{align}

Concerning the term $I_1$ we recall the idea from the proof of \cite[Theorem~ 5.3]{BisiCanizoLodsUniqueness} that for all $h\in L^1_v(\vv^q m)$ it holds
\begin{align*}
I_1
   &=\int_{\R^3}\cL^+\left(\mathbbm{1}_{\{|v|>\delta^{-1}\}}|h|\right)(v)\vv^q m(v)dv\\
   &=\int_{\R^3}\vv^q m(v)\int_{\{|v_*|>\delta^{-1}\}}k_{e}(v,v_*)|h(v_*)|dv_*dv\\
   &\leq \int_{\{|v_*|>\delta^{-1}\}}|h(v_*)|H(v_*)dv_*,
\end{align*}
where $ H(v_*)=\int_{\R^3} k_{e}(v,v_*)\vv^q m(v)dv$ for every $ v_*\in\R^3$. Hence, using Proposition \ref{prop.BoundH}, there exists some positive constant $K$ such that 
\begin{equation}\label{eqn.QeI1}
 I_1\leq K\int_{\{|v_*|>\delta^{-1}\}}|h(v_*)|(1+|v_*|^{1-\beta})\left\langle v_*\right\rangle^q m(v_*)dv_*.\\
\end{equation}

Let us recall the by \eqref{eqn.ColFreqIneq} we have $0<\nu_{e,0}|v|\leq \nu_{e,0}\vv\leq \nu_e(v)$. Therefore, using \eqref{eqn.QeI1}, we have that  
\begin{align*}
 I_1 -\int_{\R^3}&\nu_e(v)|h(v)|\vv^q m(v)dv\\
    \leq &K\int_{\{|v_*|>\delta^{-1}\}}|h(v_*)|(1+|v_*|^{1-\beta})\left\langle v_*\right\rangle^q m(v_*)dv_*\\ 
    &-\nu_{e,0}\int_{\R^3}\vv|h(v)|\vv^q m(v)dv\\
    \leq &-\nu_{e,0}\int_{\{|v|\leq\delta^{-1}\}}|h(v)|\vv^{q+1} m(v)dv\\
    &+\int_{\{|v|>\delta^{-1}\}}|h(v)|\left(K(1+|v|^{1-\beta})-\nu_{e,0}|v|\right)\vv^q m(v)dv.
\end{align*}
We claim that, since $\beta>0$ there exists $\delta_0$ sufficiently small such that 
\[
 K(1+|v|^{1-\beta})-\nu_{e,0}|v|\leq -\frac{\nu_{e,0}}{2}\vv,
\]
for every $|v|>\delta^{-1}$ with $0<\delta<\delta_0$. Indeed, if we take $\delta_0$ small enough so 
\[
 \frac{\delta_0+\delta_0^{\beta}}{1-\delta_0}\leq \frac{\nu_{e,0}}{2K},
\]
we have that for every $|v|>\delta^{-1}$ with $0<\delta<\delta_0$, since $|v|-\vv\geq -1$,
\[
 K(1+|v|^{1-\beta})\leq \nu_{e,0}(|v|-1)\leq \frac{\nu_{e,0}}{2}(2|v|-\vv),
\]
and we can conclude our claim. Therefore,
\begin{equation}\label{eqn.QeI1Nu0}
 \int_{\T^3} I_1 dx-\int_{\T^3\times\R^3}\nu_e(v)|h(v)|\vv^q m(v)dxdv\leq -\frac{\nu_{e,0}}{2}\|h\|_{L^1_xL^1_v(\vv^{q+1}m)}.
\end{equation}
Gathering \eqref{eqn.QeI2}, \eqref{eqn.QeI4} \eqref{eqn.QeI3} and \eqref{eqn.QeI1Nu0} we obtain that for any $0<\delta<\delta_0$
\begin{align*}
 \int_{\R^3\times\T^3}&\cBed(h)\text{sign}(h)\vv^q m(v)dxdv\\
 &\leq \left((C_2+C_3)\delta+C_4\Lambda(\delta)-\frac{\nu_{e,0}}{2}\right)\|h\|_{L^1_xL^1_v(\vv^{q+1}m)}.
\end{align*}
Hence, we choose $0<\delta_1\leq \delta_0$ close enough to $0$ in order to have
\begin{equation}\label{eqn.a0}
 a'_0:=-\left((C_2+C_3)\delta_1+C_4\Lambda(\delta_1)-\frac{\nu_{e,0}}{2}\right)>0.\\
\end{equation}
 
Therefore, for any $0<\delta<\delta_1$, we have 
\[
 \int_{\R^3\times\T^3}\cBed(h)\text{sign}(h)dx\vv^q m(v)dv\leq -a'_0\|h\|_{L^1_xL^1_v(\vv^{q+1} m)},
\]
where we deduce that $\cBed+a'_0$ is dissipative in $L^1_xL^1_v(\vv^{q} m)$.\\

\textbf{Step 2:} Since the $x$-derivatives commute with $\cBed$, using the proof of \textbf{Step 1} we have
\[
 \int_{\R^3\times\T^3}\partial_x\left(\cBed( h)\right)\text{sign}(\partial_x h)\vv^q m(v)dxdv\leq -a'_0\|\nabla_x h\|_{L^1_xL^1_v(\vv^{q+1} m)}.\\
\]

\textbf{Step 3:} In order to deal with the $v$-derivatives, let us recall the following property
\begin{equation}\label{eqn.Vderivative}
\partial_v\cQe^{\pm}(f,g)=\cQe^{\pm}(\partial_vf,g)+\cQe^{\pm}(f,\partial_vg).
\end{equation}
Using this and the fact that $\cLR=\cL^+(h)-\cAed(h)$ we compute
\[
 \partial_v\cBed(h)=\left(\cL^+(\partial_v h)-\cAed(\partial_v h)-\nu_e\cdot \partial _v h\right)+\cR(h),\\
\]
where 
\begin{equation}\label{eqn.DefReh}
\cR(h)=\cQe^+(h,\partial_v\cM_0)-\partial_v \cAed(h)+\cAed(\partial_v h). 
\end{equation}

Performing one integration by parts and using Lemma \ref{lem.BoundednessAed}, we have
\[
 \left\|\left(\partial_v\cAed\right)(h)\right\|_{L^1_xL^1_v(\vv^{q} m)}+ \left\|\left(\cAed\right)(\partial_v h)\right\|_{L^1_xL^1_v(\vv^{q} m)}\leq C'_{\delta} \|h\|_{L^1_xL^1_v(\vv^{q} m)}.
\]
Using this and the estimates of Proposition \ref{prop.MM31} on the operator $\cQe^+$ we obtain, for some constant $C_{\delta}>0$
\begin{equation}\label{eqn.Reh}
 \|\cR(h)\|_{L^1_xL^1_v(\vv^{q} m)}\leq C_{\delta}\|h\|_{L^1_xL^1_v(\vv^{q+1} m)}.\\
\end{equation}

Hence, by the proof presented in \textbf{Step 1} and \eqref{eqn.Reh} we have 
\begin{align*}
 \int_{\R^3\times\T^3}&\partial_v\left(\cBed( h)\right)\text{sign}(\partial_v h)\vv^q m(v)dv\\
 &= \int_{\R^3\times\T^3}\left(\cBed(\partial_v h)+\cR(h)\right)\text{sign}(\partial_v h)\vv^q m(v)dv\\
 &\leq -a'_0\|\nabla_v h\|_{L^1_xL^1_v(\vv^{q+1} m)}+C_{\delta}\|h\|_{L^1_xL^1_v(\vv^{q+1} m)}.
\end{align*}
where $a'_0$ is defined in (\ref{eqn.a0}). Notice that here $\delta$ is fixed small enough to guarantee that $a'_0>0$.\\

\textbf{Step 4:} For some $\varepsilon>0$ to be fixed later, we define the norm
\[
 \|h\|_*=\|h\|_{L^1_xL^1_v(\vv^{q} m)}+\|\nabla_x h\|_{L^1_xL^1_v(\vv^{q} m)}+\varepsilon\|\nabla_v h\|_{L^1_xL^1_v(\vv^{q} m)}.
\]
Notice that this norm is equivalent to the classical $W^{1,1}_{x,v}(\vv^q m)$-norm. Furthermore, we deduce that
\begin{align*}
 \int_{\R^3\times\T^3}&\cBed( h)\text{sign}(h)\vv^q m(v)dxdv\\
 &+\int_{\R^3\times\T^3}\partial_x\left(\cBed( h)\right)\text{sign}(\partial_x h)\vv^q m(v)dxdv\\
 &+\varepsilon\int_{\R^3\times\T^3}\partial_v\left(\cBed( h)\right)\text{sign}(\partial_v h)\vv^q m(v)dxdv\\
 &\leq -a'_0\left(\|h\|_{L^1_xL^1_v(\vv^{q+1} m)}+\|\nabla_x h\|_{L^1_xL^1_v(\vv^{q+1} m)}+\varepsilon\|\nabla_v h\|_{L^1_xL^1_v(\vv^{q+1} m)}\right)\\
 &+\varepsilon\left(C_{\delta}\|h\|_{L^1_xL^1_v(\vv^{q+1} m)}\right)\\
 &\leq (-a'_0+o(\varepsilon))\left(\|h\|_{L^1_xL^1_v(\vv^{q+1} m)}+\|\nabla_x h\|_{L^1_xL^1_v(\vv^{q+1} m)}\varepsilon\|\nabla_v h\|_{L^1_xL^1_v(\vv^{q+1} m)}\right)\\
 &\leq (-a'_0+o(\varepsilon))\|h\|_*,
\end{align*}
where $o(\varepsilon)=\varepsilon\cdot C_{\delta}$ and goes to 0 as $\varepsilon$ goes to 0 and we choose $\varepsilon$ close enough to 0 so that $a_0=a'_0-o(\varepsilon)>0$. Hence, we obtain that $\cBed+a_0$ is dissipative in $W^{1,1}_{x,v}(\vv^q m)$ for the norm $\|\cdot\|_*$ and thus hypodissipative in $W^{1,1}_{x,v}(\vv^q m)$.\\
\end{proof} 

The following result comprises what we proved above:

\begin{lemma}\label{lem.HypodissBedEj}
For any $e\in(0,1]$, there exist $a_0>0$ and $\delta>0$ such that the operator $\cBed+a_0$ is hypodissipative in $E_j$, $j = 0, 1$ and $\cE$.
\end{lemma}

\section{Splitting of the linearized elastic operator $\sLu$}\label{sec.ElastLinOp}

We now focus on the study of the linear equation $\partial_t h=\sLu(h)$ introduced in \eqref{eqn.LinearEqn} for $\alpha=1$. First of all let us recall the definition of $\sL_1$ given by \eqref{eqn.InitialValueProblem}:
\begin{equation}\label{eqn.ElasticLinearOp}
 \sL_1(h):=\cQu(\cM,h)+\cQu(h,\cM)+\cL(h)-v\cdot\nabla_x h,
\end{equation}
where $\cM$ is defined in \eqref{eqn.DefM}.\\

In the elastic case we can define the collision operator in strong form
\[
 \cQu(f,g)=\int_{\R^3\times \S^2} [f(v')g(v'_*)-f(v)g(v_*)]|v-v_*|dv_* d\sigma.
\]
Therefore, using the function $\Td$ defined in Section \ref{sec.ForcingSplit} we can give the following decomposition of the linearized collision operator $\cTu(h):=\cQu(\cM,h)+\cQu(h,\cM)$:
\[
 \cTu(h)=\cTR(h)+\cTS(h)-\nu h,
\]
where $\nu$ is a collision frequency defined in a similar way as $\nu_e$. More precisely, $\nu:=L(\cM)$, with $L$ defined as in \eqref{eqn.DefColFreqNuE}. $\cTS$ is the truncated operator given by $\Theta_{\delta}$ while $\cTR$ is the corresponding reminder. By this we mean
\[
\cTS(h)=\int_{\R^3\times\S^2}\Theta_{\delta}\left[\cM(v'_*)h(v')+\cM(v')h(v'_*)-\cM(v)h(v_*)\right]|v-v_*|dv_*d\sigma
\]
and,
\begin{equation}\label{eqn.DefTuR}
 \cTR(h)= \int_{\R^3\times\S^2}\left(1-\Theta_{\delta}\right)\left[\cM(v'_*)h(v')+\cM(v')h(v'_*)-\cM(v)h(v_*)\right]|v-v_*|dv_*d\sigma.
\end{equation}

Hence we can give a decomposition for the linearized operator $\sLu$:
\begin{align*}
\sLu(h)&=\cTu(h)+\cL(h)-v\cdot\nabla_x h\\
       &=\cTR(h)+\cTS(h)-(\nu+\nu_e)h+\cLR(h)+\cLS(h)-v\cdot\nabla_x h\\
       &=\left(\cTS(h)+\cAed(h)\right)+\left(\cTR(h)-\nu h+\cBed(h) -v\cdot\nabla_x h\right)\\
       &=\cAud(h)+\cBud(h).
\end{align*}
where $\cAud(h):=\cTS(h)+\cAed(h)$ and $\cBud(h)$ is the remainder.\\

As in the case of the operator $\cL$, by the Carleman representation for the elastic case ( see \cite[Theorem~5.4]{BisiCanizoLodsUniqueness}, \cite[Chapter~1]{Villani} or \cite[Appendix C]{GambaPanferovVillaniMaxwellian}), we can write the truncated operator 
\begin{equation}\label{eqn.TruncatedOpAud}
 \cAud(h)(v)=\int_{\R^3}k_{\delta}(v,v_*)h(v_*)dv_*,
\end{equation}
where $k_{\delta}=k_{1,\delta}+k_{e,\delta}\in C^{\infty}_c(\R^3\times\R^3)$, where $k_{1,\delta}$ is the kernel associated to the elastic operator $\cTS$ and $k_{e,\delta}$ is defined by \eqref{eqn.CalermanAed}.\\

Gualdani, Mischler and Mouhot \cite[Lemma~4.16]{GualdaniMischlerMouhot} proved a regularity estimate on the truncated operator $\cTS$. Gathering this result with the one presented in Lemma \ref{lem.BoundednessAed} we get the following regularity result for $\cAud$.

\begin{lemma}\label{lem.BoundednessAud}
For any $s\in\N$ and any $e\in(0,1]$, the operator $\cAud$ maps $L^1_v (\vv)$ into $H_v^s$ functions with compact support, with explicit bounds (depending on $\delta$) on the $L^1_v(\vv)\mapsto H_v^s$ norm and on the size of the support. More precisely, there are two constants $C_{s,\delta}$ and $R_{\delta}$ such that for any $h\in L^1_v(\vv)$ 
\[
K:= \supp \cAud h\subset B(0,R_{\delta}),\quand \|\cAud\|_{H^s_v(K)}\leq C_{s,\delta} \|h\|_{L^1_v(\vv)}.
\]
In particular, we deduce that $\cAud$ is in $\sB(E_j )$ for $j =0, 1$ and $\cAud$ is in $B(\cE, E)$.
\end{lemma}

Furthermore, Gualdani, Mischler and Mouhot \cite[Lemma~4.14]{GualdaniMischlerMouhot} study the hypositivativity of the operator $\cBud^1:=\cTR(h)-\nu h-v\cdot\nabla_x h$. More precisely, they proved that there exist a constant $\lambda_0>0$ such that $\cBud^1+\lambda_0$ is hypodissivative in $E_j$ with $j=0,1$ and in $\cE$. Thus, following their proof jointly with the one of Lemma \ref{lem.HypodissBedEj}, we are led to the following:

\begin{lemma}\label{lem.HypodissBudEj}
For any $e\in(0,1]$, there exist $a_1>0$ and $\delta>0$ such that the operator $\cBud+a_1$ is hypodissipative in $E_j$, $j =0, 1$ and $\cE$.
\end{lemma}

\begin{proof}
It is enough to see that, by the definition of $\cBed$
\[
 \cBud(h)=\left(\cTR(h)+\cLR(h)\right)-(\nu+\nu_e)h-v\cdot\nabla_x h.
\]
Notice the divergence structure of the last term in the $x$-coordinate. Hence, when integrating over $\T^3$ the last term vanishes. Therefore, proceeding as in the proof of Lemma \ref{lem.HypodissBed} we obtain our result.
\end{proof}

\subsection{Regularization properties of $T_n$}

In this section we prove the key regularity result for our factorization and enlargement theory. Consider the operators
\[
 T_n(t):=(\cAud S_{\cBud})^{(*n)}(t),
\]
for $n\geq 1$, where $S_{\cBud}$ is the semigroup generated by the operator $\cBud$ and $*$ denotes the convolution. We remind the reader that the $T_n(t)$ operators are merely time-indexed family of operators which do not have the semigroup property in general.\\

\begin{remark}
As we will see below, thanks to Theorem \ref{thm.SpectrumQe} and Theorem \ref{thm.Gualdani}, the operator $\sL_1$ generates a $C_0$-semigroup in $H_{x}^{s}H_{v}^{\sigma}(\cM^{-1})$. Therefore, jointly with the hypodissipativity of $\cBud$ this guarantees that the operator $\cBud$ also generates a $C_0$-semigroup in $E_j$ for $j=0,1$. We refer the reader to the proof of \cite[Theorem~ 2.13]{GualdaniMischlerMouhot} for a further discussion about this matter. Moreover, a direct proof can be performed as the one presented in Appendix \ref{sec.Semi} for a similar operator.\\
\end{remark}

Thanks to Lemma \ref{lem.BoundednessAud} we know that the operator $\cAud$ provides as much regularity as we want in the $v$-coordinate. By using the propagation time-dependent phase space regularity, thanks to the introduction of the operator $D_t$ below, one can still keep track of some velocity regularity, while preserving at the same time the correct time decay asymptotics.\\

\begin{lemma}\label{lem.BoundTn}
Let us consider $a_1$ as in Lemma \ref{lem.HypodissBudEj}. The time indexed family $T^n$ of operators satisfies the following: for any $a'\in (0, a_1)$ and any $e\in(0,1]$, there are some constructive constants $C_{\delta} > 0$ and $R_{\delta}$ such that for any $t\geq 0$
\[
 \supp T_n(t)h\subset K:=B(0,R_{\delta}),
\]
and
\begin{itemize}
\item If $s\geq 1$ then
\begin{equation}\label{eqn.IneqT1}
 \|T_1(t)h\|_{W^{s+1,1}_{x,v}(K)}\leq C\frac{e^{-a't}}{t}\|h\|_{W^{s,1}_{x,v}(\vv m)}.
\end{equation}
\item If $s\geq 0$ then
\begin{equation}\label{eqn.IneqT2}
 \|T_2(t)h\|_{W^{s+1/2,1}_{x}W^{k,1}_v(K)}\leq C e^{-a't}\|h\|_{W^{s,1}_{x,v}(\vv m)}.
\end{equation}
\end{itemize}
\end{lemma}

\begin{proof}
Let us consider $h_0\in W^{s,1}_{x,v}(\vv m)$, $s\in\N$. Since the $x$-derivatives commute with $\cAud$ and $\cBud$, using Lemma \ref{lem.BoundednessAud} we have
\begin{align*}
 \|T_1(t)h_0\|_{W^{s,1}_xW^{s+1,1}_{v}(K)}&=\|\cAd S_{\cBd}(t)h_0\|_{W^{s,1}_xW^{s+1,1}_{v}(K)}\\
 &\leq C\|S_{\cBud}(t)h_0\|_{W^{s,1}_{x,v}(K)}.
\end{align*}

Since $\cBud+a_1$ is hypodissipative in $W^{s,1}_{x}L^{1}_v(\vv m)$ from Theorem \ref{thm.HypodEquivalence} we have
\begin{equation}\label{eqn.BoundT1}
 \|T_1(t)h_0\|_{W^{s,1}_xW^{s+1,1}_{v}(K)}\leq Ce^{-a_1 t}\|h_0\|_{W^{s,1}_{x,v}(\vv m)}.\\
\end{equation}

Now let us assume $h_0\in W^{s,1}_xW^{s+1,1}_{v}(\vv m)$ and consider the function $g_t=S_{\cBud}(t)(\partial_x^\beta h_0)$, for any $|\beta|\leq s$. Notice that, since $\partial_t(S_{\cBud})=\cBud S_{\cBud}$, this operator satisfies
\[
 \partial_t g_t=\sL_1 g_t-\cAud g_t=\cTu(g_t)+\cL(g_t)-v\cdot\nabla_x g_t-\cAud g_t.
\]
Introducing the differential operator $D_t:=t\nabla_x + \nabla_v$, we observe that $D_t$ commutes with the free transport operator $\partial_t + v\cdot \nabla_x$, so that using the property \eqref{eqn.Vderivative} we have
\begin{align*}
 \partial_t(D_t g_t)+v\cdot\nabla_x(D_t g_t)&=D_t\left(\partial_t g_t+v\cdot\nabla_x g_t\right)\\
 &=\cTu(D_t g_t)+\cL(D_t g_t)+\cQu(g_t,\nabla_v\cM)+\cQu(\nabla_v\cM,g_t)\\
 &+\cQe(g_t,\nabla_v\cM_0)-D_t(\cAud g_t).
\end{align*}
Using the notation in \eqref{eqn.TruncatedOpAud} and performing integration by parts we get
\begin{align*}
 D_t(\cAud g_t)&=\int_{\R^3}k_{\delta}(v,v_*)t\nabla_x g_t(v_*)d_{v_*}+\int_{\R^3}\left(\nabla_v k_{\delta}(v,v_*)\right)g_t(v_*)d_{v_*}\\
 &-\int_{\R^3}k_{\delta}(v,v_*)\nabla_{v_*} g_t(v_*)d_{v_*}\\
 &=\int_{\R^3}k_{\delta}(v,v_*)(D_t g_t)(v_*)d_{v_*}+\int_{\R^3}\left(\nabla_v k_{\delta}(v,v_*)\right)g_t(v_*)d_{v_*}\\
 &+\int_{\R^3}\left(\nabla_{v_*}k_{\delta}(v,v_*)\right) g_t(v_*)d_{v_*}\\
 &=\cAud(D_t g_t)+\cAud^1(g_t)+\cAud^2(g_t).
\end{align*}
All together, we may write $ \partial_t(D_t g_t)=\cBud(D_t g_t)+\cId(g_t)$, where
\begin{align*}
 \cId(g_t)=&\cQu(g_t,\nabla_v\cM)+\cQu(\nabla_v\cM,g_t)+\cQe(g_t,\nabla_v\cM_0)\\
           &-\cAud^1(g_t)-\cAud^2(g_t).
\end{align*}

Hence, since $\cAud^1$ stands for the integral operator associated with the kernel $\nabla_v k_{\delta}$ and $\cAud^2$ stands for the integral operator associated with the kernel $\nabla_{v^*} k_{\delta}$, using Lemma \ref{lem.BoundednessAud} and Proposition \ref{prop.MM31} $\cId$ satisfies
\[
 \|\cId(g_t)\|_{L^1_v(\vv m)}\leq C\|g_t\|_{L^1_v(\vv^2 m)}.
\]
Arguing as in Lemma \ref{lem.HypodissBed}, and by the hypodissipativity of $\cBud$, we have
\begin{align}\label{eqn.DerivativeDtL1}
\frac{d}{dt}&\int_{\R^3\times\T^3}|D_t g_t|\vv m(v)dxdv\\\nonumber
&=\int_{\R^3\times\T^3}\partial_t(D_t g_t)\text{ sign}(D_t g_t)\vv m(v)dxdv\\\nonumber
&\leq\int_{\R^3\times\T^3}\cBud(D_t g_t)\text{ sign}(D_t g_t)\vv m(v)dxdv+\|\cId(g_t)\|_{L^1_x L^1_v(\vv m)}\\\nonumber
&\leq -a_1\int_{\R^3\times\T^3}|D_t g_t|\vv^2 m(v)dxdv+C_{\delta}\|g_t\|_{L^1_x L^1_v(\vv^2 m)},
\end{align}
and
\begin{align}\label{eqn.DerivativegtL1}
\frac{d}{dt}\int_{\R^3\times\T^3}|g_t|\vv m(v)dxdv&=\int_{\R^3\times\T^3}\partial_t(g_t)\text{ sign}(g_t)\vv m(v)dxdv\\\nonumber
&\leq\int_{\R^3\times\T^3}\cBud(g_t)\text{ sign}(g_t)\vv m(v)dxdv\\\nonumber
&\leq -a_1\int_{\R^3\times\T^3}|g_t|\vv^2 m(v)dxdv.
\end{align}
Combining the differential inequalities \eqref{eqn.DerivativeDtL1} and \eqref{eqn.DerivativegtL1} we obtain, for any $a'\in(0,a_1]$ and for $\varepsilon$ small enough
\[
 \frac{d}{dt}\left(e^{a't}\int(\varepsilon|D_t g_t|+|g_t|)\vv m(v)dxdv\right)\leq0,
\]
which implies
\begin{equation}\label{eqn.BoundDtgtL1}
 \|D_t g_t\|_{L^1_{x,v}(\vv m)}+\|g_t\|_{L^1_{x,v}(\vv m)}\leq \varepsilon^{-1}e^{-a't}\|h_0\|_{W^{s,1}_x W^{1,1}_v(\vv m)}.
\end{equation}

Now, notice that since the $x$-derivatives commute with $\cAud$ and $\cBud$ we have
\begin{align*}
 t\nabla_x T_1(t)\left(\partial_x^{\beta}h_0\right)&=\cAud\left(t\nabla_x g_t\right)\\
 &=\int_{\R^3}k_{\delta}(v,v_*)\left(D_t g_t\right)(v_*)dv_*+\int_{\R^3}\left(\nabla_{v_*}k_{\delta}(v,v_*)\right)g_t(v_*)dv_*\\
 &=\cAud(D_t g_t)+\cAud^2(g_t).
\end{align*}

Thus, using \eqref{eqn.BoundDtgtL1} we get
\begin{align*}
 t\|\nabla_x T_1(t)\left(\partial_x^{\beta}h_0\right)\|_{L^1_{x,v}(K)}&\leq C_{\delta}\left(\|D_t g_t\|_{L^1_{x,v}(\vv m)}+\|g_t\|_{L^1_{x,v}(\vv m)}\right)\\
 &\leq C_{\delta}\varepsilon^{-1}e^{-a't}\|h_0\|_{W^{s,1}_x W^{1,1}_v(\vv m)}.
\end{align*}
Since $0<a'\leq a_1$, using the last inequality together with \eqref{eqn.BoundT1} and Lemma \ref{lem.BoundednessAud}, for $s\geq0$ we have
\[
 \|T_1(t)\left(\partial_x^{\beta}h_0\right)\|_{W^{1,1}_x W^{s+1,1}_v(K)}\leq \frac{Ce^{-a't}}{t}\|h_0\|_{W^{s,1}_x W^{1,1}_v(\vv m)}.
\]
This concludes the proof of \eqref{eqn.IneqT1}.\\

In order to prove \eqref{eqn.IneqT2} we interpolate between the last inequality for a given $s$ i.e.
\[
 \|T_1(t)h_0\|_{W^{s+1,1}_{x,v}(K)}\leq \frac{Ce^{-a't}}{t}\|h_0\|_{W^{s,1}_x W^{1,1}_v(\vv m)},
\]
and
\[
 \|T_1(t)h_0\|_{W^{s,1}_x W^{s+1,1}_v(K)}\leq Ce^{-a_1t}\|h_0\|_{W^{s,1}_x W^{1,1}_v(\vv m)},
\]
obtained in \eqref{eqn.BoundT1} written for the same $s$, it gives us

\begin{align}\label{eqn.BoundT1WsHalf}
 \|T_1(t)&h_0\|_{W^{s+1/2,1}_{x,v}(K)}\\\nonumber
 &\leq C\left(\frac{e^{-a't}}{t}\right)^{1/2}\left(e^{-a_1 t}\right)^{1/2}\|h_0\|_{W^{s,1}_x W^{1,1}_v(\vv m)}\\\nonumber
 &\leq C\frac{e^{-a't}}{\sqrt{t}}\|h_0\|_{W^{s,1}_x W^{1,1}_v(\vv m)}.
\end{align}
Putting together \eqref{eqn.BoundT1} and \eqref{eqn.BoundT1WsHalf} for $s\geq0$ we get
\begin{align*}
\|T_2&(t)h_0\|_{W^{s+1/2,1}_{x,v}(K)}\\
&= \int_0^t\|T_1(t-\tau)T_1(\tau)h_0\|_{W^{s+1/2,1}_{x,v}(K)}d\tau\\
&\leq C\int_0^t\frac{e^{-a'(t-\tau)}}{\sqrt{t-\tau}}\|T_1(\tau)h_0\|_{W^{s,1}_x W^{1,1}_v(\vv m)}d\tau\\
& \leq C\left(\int_0^t \frac{e^{-a'(t-\tau)}}{\sqrt{t-\tau}}e^{-a_1\tau}d\tau\right)\|h_0\|_{W^{s,1}_{x,v}(\vv m)}\\
& \leq Ce^{-a't}\left(\int_0^t\frac{e^{-(a_1-a')\tau)}}{\sqrt{t-\tau}}d\tau\right)\|h_0\|_{W^{s,1}_{x,v}(\vv m)}\\
&\leq C'e^{-a't}\|h_0\|_{W^{s,1}_{x,v}(\vv m)}.
\end{align*}
This concludes our proof.\\
\end{proof}

Combining Lemma \ref{lem.HypodissBudEj} and Lemma \ref{lem.BoundTn} we get the assumptions of Lemma \ref{lem.Enlargement}. We have thus proved the following result:

\begin{lemma}\label{lem.TnNormBound}
Let us consider $a_1$ as in Lemma \ref{lem.HypodissBudEj}. For any $a'\in(0,a_1)$, there exist some constructive constants $n\in\N$ and $C_{a'}\geq 1$ such that for all $t\geq0$ and
\[
 \|T_{n}(t)\|_{\sB(E_0,E_{1})}\leq C_{a'}e^{-a't}.
\]
\end{lemma}

\subsection{Semigroup spectral analysis of the linearized elastic operator}

This section will be devoted to prove some hypodissipative results for the semigroup associated to the linearized elastic Boltzmann equation. Namely, we are going to prove the localization of the spectrum of the operator $\sL_1$, as well as the decay estimate for the semigroup. This result will be used to prove similar properties for the spectrum of $\cLa$.\\

Let us first recall an important result due to Gualdani, Mischler, Mouhot \cite{GualdaniMischlerMouhot}.\\

\begin{theorem}{\cite[Theorem~4.2]{GualdaniMischlerMouhot}}\label{thm.Gualdani}
Consider the operator 
\[
 \hat{\sL}_1(h):=\cTu(h)-v\cdot\nabla_x.
\]
Let $\cE'=H_{x}^{s}H_{v}^{\sigma}(\cM^{-1/2})$ where $s,\sigma\in\N$ with $\sigma\leq s$. Then there exist constructive constants $C\geq 1$, $\lambda > 0$, such that the operator $\hat{\sL}_1$ satisfies in $\cE'$:
\begin{align*}
&\Sigma(\hat{\sL}_1)\subset\{z\in\C:\Re (z)\leq -\lambda\}\cup\{0\}\\
&N(\hat{\sL}_1)=Spam\{\cM,v_1\cM,v_2\cM,v_3\cM,|v|^2\cM\}.
\end{align*}
It is also the generator of a strongly continuous semigroup $h_t=S_{\hat{\sL}_1}(t)h_{in}$ in $\cE'$, solution to the initial value problem $\partial_t h=\hat{\sL}_1(h)$, which satisfies for every $t\geq0$:
\[
 \|h_t-\Pi h_{in}\|_{\cE'}\leq C e^{-\lambda t}\|h_{in}-\Pi h_{in}\|_{\cE'},
\]
where $\Pi$ stands for the projection over $N(\hat{\sL}_1)$. Moreover $\lambda$ can be taken equal to the spectral gap of $\hat{\sL}_1$ in $H^s(\cM^{-1/2})$ with $s\in\N$ as large as wanted.\\
\end{theorem}

We can now formulate the main result of this section.\\

\begin{theorem}\label{thm.SpectralGapL1}
For any $e\in(0,1]$, there exist constructive constants $C\geq1$, $a_2>0$ such that the operator $\sL_1$ satisfies in $E_0$ and $E_1$:
\[
 \Sigma(\sL_1)\cap \Delta_{-a_2}=\{0\}\quand N(\sL_1)=Span\{\cM\}.
\]
Moreover, $\sL_1$ is the generator of a strongly continuous semigroup $h(t)=S_{\sL_1}h_{in}$ in $E_0$ and $E_1$, solution to the initial value problem \eqref{eqn.LinearEqn} with $\alpha=1$, which satisfies that for all $t\geq 0$ and $j=0,1$:
\[
 \left\|S_{\sL_1}(t)(\id-\PLzz)\right\|_{\sB(E_j)}\leq Ce^{-a_2t}
\]

\end{theorem}

\begin{proof}
The idea of the proof consist in deducing the spectral properties in $E_j$ from the much easier spectral analysis in $H^{s'}_{x,v}(\cM^{-1/2})$. More precisely, we will see that the assumptions of a more abstract theorem regarding the enlargement of the functional space semigroup decay are satisfied. We state this result in Theorem \ref{thm.Enlargement}.\\

Consider $\cE'=E_j$ and $E'=H^{s'}_{x,v}(\cM^{-1/2})$  with $s'$ large enough so $E'\subset\cE'$. The assumptions in \ref{it.SGGualdani2} in Theorem \ref{thm.Enlargement} are a direct consequence of the Lemmas \ref{lem.HypodissBudEj}, \ref{lem.BoundednessAud} and \ref{lem.TnNormBound}. Indeed, from Lemma \ref{lem.BoundTn} and Lemma \ref{lem.Enlargement} we have for instance
\[
 \|T_n(t)h\|_{H^{s'}_{x,v}}\leq C e^{-a't}\|h\|_{L^{1}_{x,v}(\vv m)},
\]
and so
\[
 \|T_{n+1}\|_{E'}\leq C e^{-a't}\|h\|_{\cE'}.\\
\]

So we are left with the task of verifying that assumption \ref{it.SGGualdani1} in Theorem \ref{thm.Enlargement} is satisfied. However, this is a direct consequence of Theorem \ref{thm.SpectrumQe} and Theorem \ref{thm.Gualdani} taking $a_2\in(0,a_1)$, where $a_1$ is given by Lemma \ref{lem.HypodissBudEj}. Therefore, the result follows by the Enlargement Theorem \ref{thm.Enlargement}.\\ 
\end{proof}


\section{Properties of the linearized operator}\label{sec.LinOp}

We now focus on the study of the linear equation 
\[
 \partial_t h=\cLa(h)=\cQa(\Fa,h)+\cQa(h,\Fa)+\cL(h)-v\cdot\nabla_x h,
\]
introduced in \eqref{eqn.LinearEqn} for $h=h(t,x,v)$ with $x\in\T^3$ and $v\in\R^3$.\\

First, we want to find a splitting of the linearized operator $\cLa=\cA_{\alpha}+\cB_{\alpha}$ where $\cA_{\alpha}$ is bounded and $\cB_{\alpha}$ hipodissipative in $E_j$ for $j=0,1$. We are able to do this by a perturbative argument around the elastic case. Once we obtain the localization in the spectrum and exponential decay of $\cLa$ in $E=E_0$ we, once again, apply the Enlargement Theorem \ref{thm.Enlargement} to obtain this properties in the larger space $\cE$.\\

\subsection{The linearized operator and its splitting.}

In this section we give a decomposition of the linear operator $\cLa$. In order to do this, for any $\delta\in(0,1)$ consider the bounded (by one) operator $\Td$ defined in Section \ref{sec.ForcingSplit}. We also need to consider the collision frequency $\Na=L(\Fa)$, where $\Fa$ is given by Theorem \ref{thm.ExistenceUniqSS} and $L$ by \ref{eqn.DefL}.\\

Let us define the operator $\cTa$ by $ \cTa(h)=\cQa(\Fa,h)+\cQa(h,\Fa)$. Therefore, using the weak formulation we have for any test function $\psi$
\begin{align*}
&\int_{\R^3}\cTa(h)\psi dv \\
&=\int_{\R^3\times\R^3\times\S^2}\Fa(v)h(v_*)|v-v_*|\left[\psi(v'_{*})+\psi(v')-\psi(v_*)-\psi(v)\right]d\sigma dv_* dv.
\end{align*}
Now we can give the following decomposition of the linearized collision operator $\cTa(h)$:
\[
 \cTa(h)=\cTaR(h)+\cTaS(h)-\Na h,
\]
where $\cTaS$ is the truncated operator given by $\Theta_{\delta}$ and $\cTaR$ the corresponding reminder. By this we mean:
\[
\cTaS(h)=\cQaS^+(h,\Fa)+\cQaS^+(\Fa,h)-\cQaS^-(\Fa,h),
\]
where $\cQaS^+$ (resp. $\cQaS^-$) is the gain (resp. loss) part of the collision operator associated to the mollified collision kernel $\Theta_{\delta} B$. More precisely, for any test function $\psi$
\begin{equation*}
\left\langle\cQaS^+(h,\Fa),\psi\right\rangle=\int_{\R^3\times\R^3\times\S^2} h(v)\Fa(v_*)\left(\Td\cdot|v-v_*|\right)\psi(v)d\sigma dv_* dv.\\
\end{equation*}
In a similar way
\[
\cTaR(h)=\cQaR^+(h,\Fa)+\cQaR^+(\Fa,h)-\cQaR^-(\Fa,h),
\]
where $\cQaR^+$ (resp. $\cQaR^-$) is the gain (resp. loss) part of the collision operator associated to the mollified collision kernel $(1-\Theta_{\delta})B$. \\

Hence we can give a decomposition for the linearized operator $\cLa$ in the following way:
\begin{align*}
\cLa(h)&=\cTaR(h)+\cTaS(h)-(\Na+\nu_e)h+\cLR(h)+\cLS(h)-v\cdot\nabla_x h\\
       &=\left(\cTaS(h)+\cAed(h)\right)+\left(\cTaR(h)-\Na h+\cBed -v\cdot\nabla_x h\right)\\
       &=\cAad(h)+\cBad(h).
\end{align*}
where $\cAad(h):=\cTaS(h)+\cAed(h)$ and $\cBad(h)$ is the remainder.\\

Moreover, by the Carleman representation for the inelastic case given by Arlotti Lods in \cite[Theorem~ 2.1]{ArlottiLods} we can write the truncated operator as
\begin{equation}\label{eqn.TruncatedOpAad}
 \cAad(h)(v)=\int_{\R^3}k_{\delta}(v,v_*)h(v_*)dv_*,
\end{equation}
where $k_{\delta}=k_{\alpha,\delta}+k_{e,\delta}$ with $k_{e,\delta}$ is defined in \eqref{eqn.CalermanAed}, and $k_{\alpha,\delta}$ is the kernel associated to $\cTaS$. Notice that by Lemma \ref{lem.BoundednessAed} we obtain a regularity estimate on the truncated operator $\cAad$. \\

\begin{lemma}\label{lem.BoundednessAad}
For any $s\in\N$, any $\alpha\in (0,1]$ and any $e\in(0,1]$, the operator $\cAad$ maps $L^1_v (\vv)$ into $H_v^{s+1}$ functions with compact support, with explicit bounds (depending on $\delta$) on the $L^1_v(\vv)\mapsto H_v^{s+1}$ norm and on the size of the support. More precisely, there are two constants $C_{s,\delta}$ and $R_{\delta}$ such that for any $h\in L^1_v(\vv)$ 
\[
K:= \supp \cAad h\subset B(0,R_{\delta}),\quand \|\cAad\|_{H^{s+1}_v(K)}\leq C_{s,\delta} \|h\|_{L^1_v(\vv)}.
\]
In particular, we deduce that $\cAad$ is in $\sB(E_j )$ for $j =0, 1$ and $\cAad$ is in $B(\cE, E)$.
\end{lemma}

\begin{proof}
It is clear that the range of the operator $\cAad$ is included into a compactly supported functions thanks to the truncation. Moreover, the bound on the size of the support is related to $\delta$.\\

Notice that, the proof of the smoothing estimate for the gain terms $\cQaS^+$ in the definition of $\cAad$ follows as in the proof of Lemma \ref{lem.BoundednessAed}. On the other hand, the regularity estimate is trivial for the loss term since we can decompose the truncation as $\Td=\Td^1(v)\Td^2(v-v_*)\Td^3(\cos \theta)$, and we can write
\begin{align*}
\left\langle \cQaS^-(\Fa,h),\psi\right\rangle&=\int_{\R^3\times\R^3\times\S^2}\Fa(v)h(v_*)\psi(v)\left(\Td|v-v_*|\right)d\sigma dv_* dv\\
  &=\int_{\R^3}\Td^1(v)\Fa(v)\left(h\ast\zeta_{\delta}\right)\psi(v)dv,
\end{align*}
where $\zeta_{\delta}:=\Td^2(v-v_*)\Td^3(\cos \theta)|v-v_*|$ which clearly has compact support.\\

Moreover, since the regularity of $\cAed$ is given by Lemma \ref{lem.BoundednessAed} we conclude our proof.
\end{proof}

Notice that from \cite[Lemma~2.6]{Tristani} we have that the operators $\cTa$ and $\cL$ are bounded from $W_{x}^{s,1}W_{v}^{k,1}(\vv^{q+1} m)$ to $W_{x}^{s,1}W_{v}^{k,1}(\vv^{q} m)$. Hence, using the fact that the operator $v\cdot\nabla_x$ is bounded from $E_1$ to $E_{0}$, we can conclude that the operator $\cLa$ is bounded from $E_1$ to $E_{0}$. \\

\subsection{Hypodissipativity of $\cBad$}

The aim of this section is to prove the hypodissipativity of $\cBad$. 

\begin{lemma}\label{lem.HypodissBad}
Let us consider $k\geq 0$, $s\geq k$ and $q\geq 0$. Let $\delta>0$ be given by Lemma \ref{lem.HypodissBudEj}. Then, there exist $\alpha_1\in (\alpha_0,1]$ and $a_3>0$ such that for any $\alpha\in[\alpha_1,1]$ and any $e\in(0,1]$, the operator $\cBad+a_3$ is hypodissipative in $W_{x}^{s,1}W_{v}^{k,1}(\vv^{q} m)$.
\end{lemma}

\begin{proof} 
As in the proof on Lemma \ref{lem.HypodissBed}, we only consider the case $W^{1,1}_{x,v}(\vv^q m)$. The higher-order cases are  treated in a similar way.\\

We want to construct a positive constant $a_3$ and a norm $\|\cdot\|_*$ equivalent to the norm on $W^{1,1}_{x,v}(\vv^q m)$ such that  
\begin{equation}\label{eqn.NormDerivativeBad}
 \Re\left\langle \psi,\cBed h\right\rangle\leq -a_3\|h\|_*,
\end{equation}
for $\psi(v):=\text{sign }h(v)\vv^{q}m(v)$.\\

We have divided the proof into four steps. The first one deals with the hypodissipativity of $\cBad$ in $L^1_xL^1_v(\vv^{q+1}m)$, while the second and third  deal with the $x$ and $v$-derivatives respectively. In the last step we construct the $\|\cdot\|_*$ norm and prove that it satisfies \eqref{eqn.NormDerivativeBad}.\\

\textbf{Step 1:} The main idea of the proof is to compare $\cBad$ with $\cBud$ defined in Section \ref{sec.ElastLinOp}. In order to do this notice that
\[
 \cBad-\cBud=\left(\cTaR-\cTR\right)-\left(\Na-\nu\right) h.\\
\]
Moreover, it is easy to see that the definition of $\cTR$ given in \eqref{eqn.DefTuR} coincides with the definition of $\cTaR$ with $\alpha=1$. Thus, we can write
\begin{align*}
 \cTaR(h)&-\cTR(h)=\left(\cQaR^+(h,\Fa)-\cQuR^+(h,\Fa)\right)+\left(\cQaR^+(\Fa,h)-\cQuR^+(\Fa,h)\right)\\
 &-\left(\cQaR^-(\Fa,h)-\cQuR^-(\cM,h)\right)+\cQuR^+(h,\Fa-\cM)+\cQuR^+(\Fa-\cM,h).
\end{align*}

Therefore, if we take $\psi(v)=\text{sign }h(v)\vv^{q}m(v)$, we have that
\begin{align}\label{eqn.Eta1}
\nonumber
\int_{\R^3}&\left(\cQaR^+(h,\Fa)-\cQuR^+(h,\Fa)\right)\psi(v)dv\\\nonumber
 &\leq \|\cQa^+(h,\Fa)-\cQu^+(h,\Fa)\|_{L^1_v(\vv^{q}m)}\\\nonumber
 &\leq p(\alpha-1)\|\Fa\|_{W^{k,1}_v(\vv^{q+1}m)}\|h\|_{L^1_v(\vv^{q+1} m)}\\
 &\leq \eta_1(\alpha)\|h\|_{L^1_v(\vv^{q+1} m)}.
\end{align}
Here we used Proposition \ref{prop.MM32}, where $p(r)$ is a polynomial going to 0 as $r$ goes to zero. Hence, $\eta_1(\alpha)\to0$ as $\alpha\to1$.\\

In the same manner we can see that there exists a function $\eta_2(\alpha)$ going to 0 as $\alpha$ goes to 1, such that
\begin{equation}\label{eqn.Eta2}
 \int_{\R^3}\left(\cQaR^+(\Fa,h)-\cQuR^+(\Fa,h)\right)\psi(v)dv\leq \eta_2(\alpha)\|h\|_{L^1_v(\vv^{q+1} m)}.\\
\end{equation}

Furthermore, since the loss term does not depend on the restitution coeffient, using Proposition \ref{prop.MM31} with $k=1$, we have that
\begin{align}\label{eqn.Eta3}
\nonumber
\int_{\R^3}&\left(\cQaR^-(\Fa,h)-\cQuR^-(\cM,h)\right)\psi(v)dv\\\nonumber
 &\leq \|\cQu^-(|\Fa-\cM|,h)\|_{L^1(\vv^{q}m)}\\\nonumber
 &\leq C_{1,m}\|\Fa-\cM\|_{L^1_v(\vv^{q+1} m)}\|h\|_{L^1_v(\vv^{q+1} m)}\\
 &=\eta_3(\alpha)\|h\|_{L^1_v(\vv^{q+1} m)}.
\end{align}
By Lemma \ref{lem.DifferenceFaM} we know that $\eta_3(\alpha)$ converges to 0 when $\alpha\to1$.\\

Analogously, applying Lemma \ref{lem.DifferenceFaM} we obtain a function $\eta_4(\alpha)$ converging to 0 when $\alpha\to1$ such that
\begin{equation}\label{eqn.Eta4}
\int_{\R^3}\left(\cQuR^+(h,\Fa-\cM)+\cQuR^+(\Fa-\cM,h)\right)\psi(v)dv\leq \eta_4(\alpha)\|h\|_{L^1_v(\vv^{q+1} m)}.\\
\end{equation}
Hence, taking $\eta_5:=\eta_1+\cdots+\eta_4$, gathering \eqref{eqn.Eta1}, \eqref{eqn.Eta2}, \eqref{eqn.Eta3} and \eqref{eqn.Eta4}  we have that
\begin{equation}\label{eqn.Eta5}
 \int_{\R^3}\left(\cTaR(h)-\cTR(h)\right)\psi(v)dv\leq \eta_5(\alpha)\|h\|_{L^1_v(\vv^{q+1} m)}.\\
\end{equation}

We have proved that there exist $\eta_5(\alpha)$ converging to 0 when $\alpha\to1$ such that
\begin{align*}
 \int_{\R^3}&\cBad(h)\text{sign}(h)\vv^q m(v)dv\\
  &\leq \int_{\R^3}\cBud(h)\text{sign}(h)\vv^q m(v)dv+\eta_5(\alpha)\|h\|_{W^{k,1}_v(\vv^{q+1}m)}\\
  &+\int_{\R^3}\left|\Na(v)-\nu(v)\right|\|h\|\vv^{q+1}m(v)dv.
\end{align*}

Finally, since $|v-v_*|\leq\vv\left\langle v_*\right\rangle$ we have
\[
 \left|\Na(v)-\nu(v)\right|\leq \int_{\R^3}|v-v_*\|\Fa(v_*)-\cM(v_*)|dv_*\leq \vv\|\Fa-\cM\|_{L^1(\vv^{q} m)}.
\]
Then, we deduce from Lemma \ref{lem.DifferenceFaM} that
\[
 \int_{\R^3}\left|\Na(v)-\nu(v)\right\|h(v)|\vv^{q}m(v)dv\leq \eta_6(\alpha)\|h(v)\|_{L^1(\vv^{q+1}m)},
\]
with $\lim_{\alpha\to1}\eta_6(\alpha)=0$. To summarize, there exists a function $\eta=\eta_5+\eta_6$ converging to 0 when $\alpha\to1$ such that 
\begin{align*}
\int_{\R^3}&\cBad(h)\text{sign}(h)\vv^q m(v)dv\\
  &\leq \int_{\R^3}\cBud(h)\text{sign}(h)\vv^q m(v)dv+\eta(\alpha)\|h\|_{L^1_v(\vv^{q+1} m)}.
\end{align*}
Thus, fixing $\delta>0$ as in Lemma \ref{lem.HypodissBudEj}, this inequality becomes
\[
 \int_{\R^3\times\T^3}\cBad(h)\text{sign}(h)\vv^q m(v)dxdv\leq (\eta(\alpha)-a_1)\|h\|_{L^1_xL^1_v(\vv^{q+1} m)}.
\]

Taking $\alpha_1$ big enough we can suppose that for any $\alpha\in[\alpha_1,1]$ we have $\eta(\alpha)<a_1$ and therefore
\begin{equation}\label{eqn.a3}
 a'_3:=a_1-\eta(\alpha)>0.
\end{equation}
 
With this we can conclude 
\[
 \int_{\R^3\times\T^3}\cBad(h)\text{sign}(h)\vv^q m(v)dxdv\leq -a'_3\|h\|_{L^1_xL^1_v(\vv^{q+1} m)},
\]
where we deduce that for any $\alpha\in[\alpha_1,1]$, $\cBad+a'_3$ is dissipative in $L^1_xL^1_v(\vv^{q+1} m)$.\\

\textbf{Step 2:} Since the $x$-derivatives commute with $\cBad$, using the proof of \textbf{Step 1} we have
\[
 \int_{\R^3\times\T^3}\partial_x\left(\cBad( h)\right)\text{sign}(\partial_x h)\vv^q m(v)dxdv\leq -a'_3\|\nabla_x h\|_{L^1_xL^1_v(\vv^{q+1} m)}.\\
\]

\textbf{Step 3:} In order to deal with the $v$-derivatives, we proceed analogously to the proof of Lemma \ref{lem.HypodissBed} to see that
\[
 \partial_v\cBad(h)=\cBad(\partial_v h)-\partial_x h+\cR'(h),
\]
where, recalling that $\cR(h)$ is given by \eqref{eqn.DefReh},  
\begin{align}\label{eqn.DefRah}
 \cR'(h)&=\cQa^+(h,\partial_v \Fa)+\cQa^+(\partial_v\Fa,h)-\cQa^-(\partial_v \Fa,h)\\\nonumber
   &-\left(\partial_v\cTaS\right)(h)+\cTaS(\partial_v h)+\cR(h).
\end{align}
Proceeding as in the proof of Lemma \ref{lem.BoundednessAad} and performing one integration by parts, we have
\[
 \left\|\left(\partial_v\cTaS\right)(h)\right\|_{L^1_xL^1_v(\vv^{q} m)}+ \left\|\cTaS(\partial_v h)\right\|_{L^1_xL^1_v(\vv^{q} m)}\leq C_{\delta} \|h\|_{L^1_xL^1_v(\vv^{q} m)}.
\]
Using this, the estimates of Proposition \ref{prop.MM31} on the inelastic operators $\cQa^{\pm}$ and the bound of $\cR(h)$ given in \eqref{eqn.Reh}, we obtain for some constant $C_{\alpha,\delta}>0$
\begin{equation}\label{eqn.Rah}
 \|\cR'(h)\|_{L^1_xL^1_v(\vv^{q} m)}\leq C_{\alpha,\delta}\|h\|_{L^1_xL^1_v(\vv^{q+1} m)}.
\end{equation}

Thus, using the bounds found in \eqref{eqn.Rah} and the proof presented in \textbf{Step 1} we have
\begin{align*}
 \int_{\R^3\times\T^3}&\partial_v\left(\cBad( h)\right)\text{sign}(\partial_v h)\vv^q m(v)dxdv\\
 &\leq -a'_3\|\nabla_v h\|_{L^1_xL^1_v(\vv^{q+1} m)}+C_{\alpha,\delta}\|h\|_{L^1_xL^1_v(\vv^{q+1} m)}+\|\nabla_x h\|_{L^1_xL^1_v(\vv^{q+1} m)}.
\end{align*}
where $a'_3$ is defined in \eqref{eqn.a3}. \\

\textbf{Step 4:} Now, for some $\varepsilon>0$ to be fixed later, we define the norm
\[
 \|h\|_*=\|h\|_{L^1_xL^1_v(\vv^{q} m)}+\|\nabla_x h\|_{L^1_xL^1_v(\vv^{q} m)}+\varepsilon\|\nabla_v h\|_{L^1_xL^1_v(\vv^{q} m)}.
\]
Notice that this norm is equivalent to the classical $W^{1,1}_{x,v}(\vv^q m)$-norm. We deduce that
\begin{align*}
\int_{\R^3\times\T^3}&\cBad( h)\text{sign}(h)\vv^q m(v)dxdv\\
 &+\int_{\R^3\times\T^3}\partial_x\left(\cBad( h)\right)\text{sign}(\partial_x h)\vv^q m(v)dxdv\\
 &+\varepsilon\int_{\R^3\times\T^3}\partial_v\left(\cBad( h)\right)\text{sign}(\partial_v h)\vv^q m(v)dxdv\\
 &\leq -a'_3\left(\|h\|_{L^1_xL^1_v(\vv^{q+1} m)}+\|\nabla_x h\|_{L^1_xL^1_v(\vv^{q+1} m)}+\varepsilon\|\nabla_v h\|_{L^1_xL^1_v(\vv^{q+1} m)}\right)\\
 &+\varepsilon\left(C_{\alpha,\delta}\|h\|_{L^1_xL^1_v(\vv^{q+1} m)}+\|\nabla_x h\|\right)\\
 &\leq (-a'_3+o(\varepsilon))\left(\|h\|_{L^1_xL^1_v(\vv^{q+1} m)}+\|\nabla_x h\|_{L^1_xL^1_v(\vv^{q+1} m)}\right.\\
 &\left.+\varepsilon\|\nabla_v h\|_{L^1_xL^1_v(\vv^{q+1} m)}\right),
\end{align*}
where $o(\varepsilon)\to 0$ as $\varepsilon$ goes to 0. We choose $\varepsilon$ close enough to 0 so that $a_3=a'_3-o(\varepsilon)>0$. Hence, we obtain that $\cBd+a_3$ is dissipative in $W^{1,1}_{x,v}(\vv^q m)$ for the norm $\|\cdot\|_*$ and thus hypodissipative in $W^{1,1}_{x,v}(\vv^q m)$.
\end{proof} 

The following result comprises what we proved above:

\begin{lemma}\label{lem.HypodissBadEj}
There exist $\alpha_1\in (\alpha_0,1]$, $\delta\geq0$ and $a_3>0$ such that for any $\alpha\in[\alpha_1,1]$ and any $e\in(0,1]$, the operator $\cBad+a_3$ is hypodissipative in $E_j$, $j = 0, 1$ and $\cE$.
\end{lemma}

\subsection{Regularization properties of $T_n$ and estimates on $\cLa-\sL_1$}

Let us recall the notation 
\[
 T_n(t):=(\cAad S_{\cBad})^{(*n)}(t),
\]
for $n\geq 1$, where $S_{\cBad}$ is the semigroup generated by the operator $\cBad$ and $*$ denotes the convolution. The proof of the fact that $\cBad$ generates a $C_0$-semigroup can be found in Appendix \ref{sec.Semi}.\\

Notice that the proof of Lemma \ref{lem.BoundTn} remains valid in this case. Therefore, combining Lemma \ref{lem.HypodissBadEj} and Lemma \ref{lem.BoundTn} we get the assumptions of Lemma \ref{lem.Enlargement}, so applying it we get the following result:

\begin{lemma}\label{lem.TnNormBoundLa}
Let us consider $\alpha_1$ and $a_3$ as in Lemma \ref{lem.HypodissBadEj} and let $\alpha$ be in $[\alpha_1,1)$. For any $a'\in(0,a_3)$ and for any $e\in(0,1]$, there exist some constructive constants $n\in\N$ and $C_{a'}\geq 1$ such that for all $t\geq0$ 
\[
 \|T_{n}(t)\|_{\sB(E_0,E_{1})}\leq C_{a'}e^{-a't}.\\
\]
\end{lemma}


Moreover, using estimates from the proof of Lemma \ref{lem.HypodissBad}, we can prove the following result:

\begin{lemma}\label{lem.BoundLaL1}
There exists a function $\eta(\alpha)$ that tends to 0 as $\alpha$ tends to 1 such that the difference
$\cLa-\sL_1$ satisfies for any $e\in(0,1]$
\[
 \|\cLa-\sL_1\|_{\sB(E_1,E_{0})}\leq \eta(\alpha).
\]
\end{lemma}

\begin{proof}
First of all notice that 
\begin{align*}
  \cLa-\sL_1&=\cTa-\cTu\\
    &=\left(\cTa^+(h)-\cTu^+(h)\right)+\cQu(h,\Fa-\cM)+\cQu(\Fa-\cM,h).
\end{align*}

Therefore, by Proposition \ref{prop.MM32} we have $\eta_1(\alpha)$ 
\begin{equation}\label{eqn.GainTaTu}
 \|\cTa^+(h)-\cTu^+(h)\|_{L^1_v(\vv^q m)}\leq\eta_1(\alpha)\|h\|_{L^1_v(\vv^{q+1} m)},
\end{equation}
with $\eta_1(\alpha)\to0$ when $\alpha\to1$. Moreover, by Proposition \ref{prop.MM31} and Lemma \ref{lem.DifferenceFaM} there exists $\eta_2(\alpha)$ such that
\[
 \|\cQu(h,\Fa-\cM)\|_{L^1_v(\vv^q m)}+\|\cQu(\Fa-\cM,h)\|_{L^1_v(\vv^q m)}\leq\eta_2(\alpha)\|h\|_{L^{1}_v(\vv^{q+1} m)},
\]
with $\eta_2(\alpha)\to0$ when $\alpha\to1$. Thus, taking $\eta'=\eta_1+\eta_2$ we have
\[
 \|\cTa-\cTu\|_{L^1_v(\vv^q m)}\leq \eta'(\alpha)\|h\|_{L^1_v(\vv^{q+1} m)}.\\
\]

Furthermore, using \eqref{eqn.Vderivative} we have that 
\begin{align*}
 \partial_v\left(\cLa^+(h)-\sL_1^+(h)\right) &=\partial_v\left(\cTa^+(h)-\cTu^+(h)\right)\\
  &=\cTa^+(\partial_v h)-\cTu^+(\partial_v h)+\left(\cQa^+(h,\partial_v\Fa)-\cQu^+(h,\partial_v\Fa)\right)\\
  &+\left(\cQa^+(\partial_v\Fa,h)-\cQu^+(\partial_v\Fa,h)\right)+\cQu(\partial_v h,\Fa-\cM)\\
  &+\cQu(\Fa-\cM,\partial_v h).
\end{align*}
Thus, by \eqref{eqn.GainTaTu} we have that
\[
 \|\cTa^+(\partial_v h)-\cTu^+(\partial_v h)\|_{L^1_v(\vv^q m)}\leq\eta_1(\alpha)\|\partial_v h\|_{L^1_v(\vv^{q+1} m)}.
\]

Proceeding as before by Proposition \ref{prop.MM32} we have that there exist $\eta_3$ converging to 0 as $\alpha$ goes to 1 such that
\begin{align*}
 \|\cQa^+(h,\partial_v\Fa)-\cQu^+(h,\partial_v\Fa)\|_{L^1_v(\vv^q m)}&+\|\cQa^+(\partial_v\Fa,h)-\cQu^+(\partial_v\Fa,h)\|_{L^1_v(\vv^q m)}\\
  &\leq \eta_3(\alpha)\|h\|_{L^1_v(\vv^{q+1} m)}.
\end{align*}
And using again Proposition \ref{prop.MM31} and Lemma \ref{lem.DifferenceFaM} there exist $\eta_4(\alpha)$ such that
\begin{align*}
 \|\cQu(h,\partial_v(\Fa-\cM))&+\cQu(\partial_v h,\Fa-\cM)\|_{L^1_v(\vv^q m)}\\
  &\leq \eta_4(\alpha)\left(\|h\|_{L^1_v(\vv^{q+1} m)}+\|\partial_v h\|_{L^1_v(\vv^{q+1} m)}\right).\\
\end{align*}

For higher-order derivatives we proceed in the same way and we can conclude that there exists some $\eta$ such that it tends to 0 as $\alpha$ tends to 1, and satisfies
\[
 \|\cLa-\sL_1\|_{E_0}\leq\eta(\alpha)\|h\|_{E_1}.
\]
In a similar way we obtain
\[
 \|\cLa-\sL_1\|_{E_{-1}}\leq\eta(\alpha)\|h\|_{E_0}.
\]
\end{proof}

For now on we fix $\delta$ as in Lemma \ref{lem.HypodissBadEj} and we write $\cA=\cAad$ and $\cB=\cBad$.

\subsection{Semigroup spectral analysis of the linearized operator}

This section is dedicated to present some results regarding the geometry of the spectrum of the linearized inelastic collision operator for a parameter close to 1. \\

\begin{proposition}\label{prop.SpectralGap}
There exists $\alpha_2\in [0,1)$ such that for any $\alpha\in[\alpha_2,1)$ and any $e\in(0,1]$, $\cLa$ satisfies the following properties in $E=W^{s,1}_xW^{2,1}_v(\vv m)$, $s\in\N^*$:
\begin{enumerate}[(a)]
\item \label{it.ThmSpectGap1}  $\Sigma(\cLa)\cap \Delta_{-a_2}=\{0\}$ where $a_2$ is given by Theorem \ref{thm.SpectralGapL1}. Moreover, 0 simple eigenvalue of $\cLa$ and $N(\cLa)=Span\{\Fa\}.$
\item \label{it.ThmSpectGap3} For any $a\in(0,\min(a_2,a_3))$, where $a_3$ is given by Lemma \ref{lem.HypodissBadEj}, the semigroup generated by $\cLa$ has the following decay property 
      \begin{equation}\label{eqn.SemigroupDecay}
       \|\SLa(t)(I-\PLaz)\|_{\sB(E)}\leq Ce^{-at},
      \end{equation}
      for all $t\geq0$ and for some $C>0$.
\end{enumerate}
\end{proposition}

The proof of the proposition stated above is a straightforward adaptation of one presented in \cite[Proposition~2.14]{Tristani}. We shall only mention the main steps of the proof and we emphasize the few points which differs here (due to the replacement of the diffusive term by a linear scattering operator). \\

\begin{proof}
\noindent\textit{\textbf{Step 1: Localization of the spectrum of $\cLa$ and dimension of eigenspaces.}} \\

Notice that, by the result \cite[Lemma~ 2.16]{Tristani} (which we get due to Lemmas \ref{lem.HypodissBadEj}, \ref{lem.BoundLaL1}, \ref{lem.TnNormBoundLa} and Theorem \ref{thm.SpectralGapL1}) we know that there exist $\alpha'>\alpha_1$ such that $\cLa-z$ is invertible for any $z\in\Oa=\Delta_{-a_2}\backslash\{0\}$ and any $\alpha\geq \alpha'$. Moreover, we have that 
\[
 \Sigma(\cLa)\cap\Delta_{-a_2}\subset B(0,\eta'(\alpha)),
\]
where $\eta'$ goes to $0$ as $\alpha$ goes to 1. Furthermore, by \cite[Lemma~ 2.17]{Tristani} which remains true in our context as a result of Lemmas \ref{lem.HypodissBadEj} and \ref{lem.BoundLaL1}, there exist a function $\eta''(\alpha)$ such that
\begin{equation}\label{eqn.BoundProjectionsDif}
 \|\PLaa-\PLza\|_{\mathfrak{B}(E_0)}\leq\eta''(\alpha),
\end{equation}
with $\eta''(\alpha)\to 0 $ as $\alpha\to 0$. Hence, by Theorem \ref{thm.SpectralGapL1}, it implies that for $\alpha$ close to 1, we have
\[
 \dim R(\PLaa)=\dim R(\PLza)=1.\\
\]

Therefore, there exist $\alpha_2>\alpha'$ such that $\eta''(\alpha)<1$ for every $\alpha\in(\alpha_2,1]$. Also there exist $\xi_{\alpha}\in\C$ such that 
\[
 \Sigma(\cLa)\cap\Delta_{-a_2}=\{\xi_{\alpha}\}.
\]
Let us prove that $\xi_{\alpha}=0$. We argue by contradiction and assume that for $\alpha$ close to 1 we have $\xi_{\alpha}\neq 0$. Let $\varphi_{\alpha}$ be some normalized eigenfunction of $\cLa$ associated to $\xi_{\alpha}$, i.e. satisfies $\cLa\varphi_{\alpha}=\xi_{\alpha}\varphi_{\alpha}$. Integrating over $\R^3$ we get that
\[
 \int_{\R^3}\varphi_{\alpha}(v)dv=0.
\] 
For any $h\in E_0$ there exist $\rho=\rho(\alpha,h)$ and $\rho'=\rho'(h)$ such that $\Pi_{\cLa,\xi_{\alpha}} h=\rho\varphi_{\alpha}$ while $\Pi_{\sL_1,0}h=\rho'\cM$. Hence, we have 
\[
 \int_{\R^3}\Pi_{\cLa,\xi_{\alpha}} hdv=0\quand \int_{\R^3}\Pi_{\sL_1,0}h dv=\rho',
\]
which contradicts \eqref{eqn.BoundProjectionsDif}. Therefore, $\xi_{\alpha}=0$. Furthermore, $0$ is a simple eigenvalue of $\cLa$ since $\Fa$ is the unique steady state of $\cLa$ satisfying $\int_{\R^3}\Fa dv=1$.\\

\noindent\textit{\textbf{Step 2: Semigroup decay.}} \\

In order to prove the estimate on the semigroup decay (\ref{eqn.SemigroupDecay}) we apply the Spectral Mapping Theorem \ref{thm.SpectralMappingThm} with $a=\max\{-a_2,-a_3\}<0$. First of all notice that $E_1\subset D(\cLa^2)\subset E_0$. Moreover, by the results presented in Lemmas \ref{lem.HypodissBadEj}, \ref{lem.BoundednessAad} and \ref{lem.TnNormBoundLa} the assumptions \ref{it.SpectralMapThmI}, \ref{it.SpectralMapThmII}, and \ref{it.SpectralMapThmIII} of Theorem \ref{thm.SpectralMappingThm} are satisfied.\\ 

Furthermore, the condition \ref{it.SpectralMapThm1} in Theorem \ref{thm.SpectralMappingThm} is also satisfied by \textbf{Step 1}. Thus we have the decay result \eqref{eqn.SemigroupDecay} for any $a'\in(0,\min\{a_2,a_3\})$. This concludes the proof of Proposition \ref{prop.SpectralGap}.\\
\end{proof}

Combining the results of Lemmas \ref{lem.HypodissBadEj}, \ref{lem.BoundednessAad} and Proposition \ref{prop.SpectralGap} we fulfilled the assumptions of Theorem \ref{thm.Enlargement}. Therefore, we have a localized spectrum and exponential decay of the semigroup of $\cLa$ in a larger space:\\

\begin{theorem}\label{thm.ExponentialDecay}
There exist $\alpha_2\in (0,1]$ such that for any $\alpha\in[\alpha_2,1)$ and any $e\in(0,1]$, $\cLa$ satisfies the following properties in $\cE=W^{s,1}_x L^1_v(m)$, $s\geq2$:
\begin{enumerate}
\item The spectrum $\Sigma(\cLa)$ satisfies the separation property: $\Sigma(\cLa)\cap \Delta_{-a_2}=\{0\}$ where $a_2$ is given by Theorem \ref{thm.SpectralGapL1} and $N(\cLa)=\{\Fa\}$.
\item For any $a\in(0,\min\{a_2,a_3\})$, where $a_3$ is provided by Lemma \ref{lem.HypodissBadEj}, the semigroup generated by $\cLa$ has the following decay property for every $t\geq0$
\begin{equation}\label{eqn.DecayEstimate}
 \left\|\SLa(t)(\id-\PLaz)\right\|_{\sB(\cE)}\leq Ce^{-at},
\end{equation}
for some $C>0$.
\end{enumerate}
\end{theorem}
\section{The nonlinear Boltzmann equation}\label{sec.Nonlinear}

Let us fix the integer $s>6$. Consider the Banach spaces 
\begin{align*}
&\cE_1:=W^{s,1}_{x}L^1_v(\vv m),\\
&\cE:=W_{x}^{s,1}L_{v}^{1}(m).
\end{align*}

Consider the following norm in $\cE$
\begin{equation}\label{eqn.BanachNorm}
 \vertiii{h}_{\cE}:=\eta\|h\|_{\cE}+\int_0^{+\infty}\|\SLa(\tau)(I-\PLaz)h\|_{\cE} d\tau,
\end{equation}
for $\eta>0$. This norm is well-defined thanks to estimate \eqref{eqn.DecayEstimate} for $\alpha$ close to 1. Furthermore, we define $\vertiii{\cdot}_{\cE_1}$ as in \eqref{eqn.BanachNorm} for the space $\cE_1$.\\

For this Banach norm the semigroup is not only dissipative, it also has a stronger dissipativity property: the damping term in the energy estimate controls the norm of the graph of the collision operator. More precisely:\\

\begin{proposition}\label{prop.NormDerivative}
Consider $\alpha\in[\alpha_2,1)$. There exist $\eta>0$ and $K>0$ such that for any initial datum $h_{in}\in\cE$ satisfying $\PLaz h_{in}=0$, the solution $h_t:=\SLa(t)h_{in}$ to the initial value problem \eqref{eqn.LinearEqn} satisfies for every $t\geq 0$
\[
 \frac{d}{dt}\vertiii{h_t}_{\cE}\leq -K\vertiii{h_t}_{\cE_1}.
\] 
\end{proposition}

Due to the dissipativity of $\cB$ proved in Lemma \ref{lem.HypodissBad} and the bounds on $\cA$ in Lemma \ref{lem.BoundednessAad},  the proof of this result follows as in \cite[Proposition~2.23]{Tristani}. A fundamental observation in our case is that 
\[
 \PLaz h_t=\PLaz \SLa h_{in}=\SLa \PLaz h_{in}=0.
\]
Therefore, Tristani's argument remains valid even though we no longer have conservation of the momentum.\\

We can now proceed to prove our main result. Namely, the existence of solutions of \eqref{eqn.BoltzmannEqn} in the close-to-equilibrium regime:

\begin{theorem}\label{thm.Main}
Consider constant restitution coefficient $\alpha\in[\alpha_0,1]$, where $\alpha_0$ is given by Theorem \ref{thm.ExponentialDecay} and any $e\in(0,1]$. There exist constructive constant $\varepsilon>0$ such that for any initial datum $f_{in}\in\cE$ satisfying
\[
 \|f_{in}-\Fa\|_{\cE}\leq\varepsilon,
\]
and $f_{in}$ has the same global mass as the equilibrium $\Fa$, there exist a unique global solution $f\in L^{\infty}_t(\cE)\cap L^1_t(\cE_1)$ to \eqref{eqn.BoltzmannEqn}.\\

Moreover, consider $a\in(0,\min\{a_2,a_3\})$, where $a_2$ is given by Theorem \ref{thm.SpectralGapL1} and $a_3$ by Lemma \ref{lem.HypodissBadEj}. This solution satisfies that for some constructive constant $C\geq 1$ and for every $t\geq 0$
\[
 \|f-\Fa\|_{\cE}\leq Ce^{-at}\|f_{in}-\Fa\|_{\cE}.
\]
\end{theorem}

\begin{proof}
Due to the study of the semigroup in Theorem \ref{thm.ExponentialDecay} one can build a solution by the use of an iterative scheme whose convergence is ensured due to a priori estimates coming from estimates of the semigroup of the linearized operator estabilshed in Proposition \ref{prop.NormDerivative}. This is a standar procedure and follows exactly as in the proof of \cite[Theorem~3.2]{Tristani}.\\
\end{proof}

\begin{remark}
The assumption $s>6$ is a technical condition that guarantees the continuous embedding of $W_x^{s/2,1}\subset L_x^{\infty}(\T^3)$ in order to have bilinear estimate \cite[Lemma~ 3.1]{Tristani}.\\
\end{remark}
\section*{Acknowledgement}

We thank R. Alonso and B. Lods for helpful discussions on the topic and the suggestion of the problem. The author was partially founded by CAPES/PROEX, PUC-Rio and Petrobras. 


\appendix

\section{Interpolation Inequalities and norm bounds}\label{sec.Inequalities}

In this section we are going to present several inequalities and norm bounds that we need along this work. Let us begin with some useful inequalities.\\

\begin{lemma}\label{lem.ExistenceCbeta}
 Let $\C_{\beta}=2^{\beta/2}-1$, then if $x\leq a/2$ we have
\[
 (a-x)^{\beta/2}\leq a^{\beta/2}-C_{\beta}x^{\beta/2}.
\]
\end{lemma}

\begin{proof}
Since $x\leq a/2$ then $2\leq y=a/x$ and
\[
 (a-x)^{\beta/2}=x^{\beta/2}(y-1)^{\beta/2}.
\]
Recall that if $0<p<1$ then for every $A,B>0$ we have that $|A^p-B^p|\leq |A-B|^p$. Thus
\[
 (A-B)^p \leq -|A^p-B^p| \leq -B^p+A^p.
\]
Hence, taking $A=y\geq 2=B$ we get
\[
 |y-1|^p-1\leq |(y-1)^p-1^p|\leq |y-2|^p\leq y^p-2^p.
\]
Therefore, if $p=\beta/2$
\[
 (a-x)^{\beta/2}\leq x^{\beta/2}(y^{\beta/2}-2^p+1),
\]
taking $C_{\beta}=2^{\beta/2}-1$ we conclude our result.\\

\end{proof}

\begin{lemma}\label{lem.ExistenceCb}
Given any $\gamma>0$ there exist a positive constant $C_{\gamma}>0$ such that
\[
 b|v_*|^{\beta}-\beta_0|v_*|^2\leq C_{\gamma}-\gamma|v_*|^{\beta}.
\]
\end{lemma}

\begin{proof}
Given $\gamma>0$ for $x\geq 0$ consider the function 
\[
 f(x)=(b+\gamma)x^{\beta}-\beta_0 x^2.
\]
It is easy to see that this function attains its maximum when $x^{2-\beta}=\beta(b+\gamma)/(2\beta_0)$. Therefore, it is enough to take $C_{\gamma}$ as this maximum.
\end{proof}

Let's recall an interpolation inequality given by \cite[Lemma~B.1]{MischlerMouhotSelfSimilarity} which can be easily extended to other weights of type $\vv^q m$.

\begin{lemma}\label{lem.Interpolation}
For any $k,q\in\N$, there exists $C>0$ such that for any $h\in H^{k'}\cap L^1(m^{12})$ with $k'=8k+7(1+3/2)$ 
\[
 \|h\|_{W^{k,1}_v(\vv^q m)}\leq C\|h\|^{1/8}_{H_v^{k'}}\|h\|^{1/8}_{L^1_v(m^{12})}\|h\|^{3/4}_{L^1_v(m)}.
\]
\end{lemma}

Now we present an useful result given by Mischler Mouhot in \cite[Proposition~3.1]{MischlerMouhotSelfSimilarity}. Although, they proved it for a different type of weights it can we extended to our case.

\begin{proposition}\label{prop.MM31}
For any $k,q\in\N$ there exist $C>0$ such that for any smooth functions $f,g$ (say $f,g\in\mathcal{S}(\R^N)$) and any $\alpha\in[0,1]$ it holds
\[
 \|\cQa^{\pm}(f,g)\|_{W^{k,1}_v(\vv^q m)}\leq C_{k,m}\|f\|_{W^{k,1}_v(\vv^{q+1} m)}\|g\|_{W^{k,1}_v(\vv^{q+1} m)}.
\]
\end{proposition}

\begin{proof}
Let us begin by considering $\cQa^-$. Recall that 
\begin{equation*}
\left\langle\cQa^-(f,g),\psi\right\rangle=\int_{\R^3\times\R^3\times\S^2}g(v_*)f(v)\psi(v)|v-v_*|d\sigma dv_* dv=\left\langle fL(g),\psi\right\rangle,
\end{equation*}
where $L(g)(v)=4\pi(|\cdot|\ast g)(v)$. Hence, since 
\[
 |L(g)|\leq 4\pi\vv\|g\|_{L^|_v(\vv)},
\]
we have that
\[
 \|\cQa^-(f,g)\|_{L^1_v(\vv^q m)}\leq  4\pi\|g\|_{L^1_v(\vv^{q+1} m)}\|f\|_{L^1_v(\vv^{q+1} m)}.
\]
Moreover, using \eqref{eqn.Vderivative} we obtain a constant $C>0$ such that
\[
 \|\partial_v\cQa^-(f,g)'\|_{L^1_v(\vv^q m)}\leq C \|g\|_{W^{1,1}_v(\vv^{q+1} m)}\|f\|_{W^{1,1}_v(\vv^{q+1} m)}.
\]
We continue in this fashion obtaining our result.\\

We proceed to estimate the gain term. By duality we have
\begin{align*}
 \|\cQa^+(f,g)&\|_{L^1_v(\vv^q m)}=\sup_{\|\psi\|_{L^{p'}}=1}\int_{\R^3}\cQa^+(f,g)\psi(v)\vv^q m(v)dv\\
 &=\sup_{\|\phi\|_{L^{p'}}=1}\int_{\R^3\times\R^3}f(v)g(v_*)|v-v_*|\int_{\S^2}\psi(v')\left\langle v'\right\rangle^q m(v')d\sigma dv_*dv\\
 &=\sup_{\|\phi\|_{L^{p'}}=1}I(\psi).
\end{align*}
Hence, using \eqref{eqn.WeightIneqAbsolute} we can assert that for some positive constant $C_1$
\[
 I(\psi)\leq C_1\int_{\R^3\times\R^3}f(v)g(v_*)\vv^{q+1}m(v) \left\langle v_* \right\rangle^q m(v_*)\int_{\S^2}\psi(v')d\sigma dv_*dv.\\
\]

Let  us call $S(\psi)=\int_{\S^2}\psi(v')d\sigma$. In fact, we split $S(\psi)$ into two parts $S_{+}(\psi)$ and $S_{-}(\psi)$ where
\[
 S_{\pm}(\psi)=\int_{\pm u\cdot\sigma>0}\psi(v')d\sigma.
\]
By \cite[Proposition~ 4.2]{GambaPanferovVillani} we have that the operators
\begin{align*}
&S_+:L^r(\R^3)\to L^{\infty}(\R^3_v,L^r(\R^3_{v_*})),\\
&S_-:L^r(\R^3)\to L^{\infty}(\R^3_{v_*},L^r(\R^3_{v}))
\end{align*}
are bounded for every $1\leq r\leq \infty$. Therefore, we conclude
\begin{align*}
 &\|\cQa^+(f,g)\|_{L^1_v(\vv^q m)}\\
 &\leq \sup_{\|\psi\|_{L^{p'}}=1}C_1\int_{\R^3\times\R^3}f(v)g(v_*)\vv^{q+1}m(v)\left\langle v_*\right\rangle^q m(v_*)(S_{+}(\psi)+S_{-}(\psi)) dv_*dv\\
 &\leq C_2\|f\|_{L^1_v(\vv^{q+1}m)}\|g\|_{L^1_v(\vv^{q+1}m)}.
\end{align*}
For the derivatives we proceed analogously to the proof for the loss term.
\end{proof}

The following result may be proved in much the same way as \cite[Proposition~3.2]{MischlerMouhotSelfSimilarity}.\\

\begin{proposition}\label{prop.MM32}
For any $\alpha,\alpha'\in(0,1]$, and any $g\in L^1_v(\vv^{q+1} m)$, $f\in W^{1,1}_v(\vv^{q+1} m)$ it holds
\[
 \|\cQa^{+}(f,g)-\mathcal{Q}_{\alpha'}^+(f,g)\|_{L^1_v(\vv^q m)}\leq p(\alpha-\alpha')\|f\|_{W^{1,1}_v(\vv^{q+1}m)}\|g\|_{L^1_v(\vv^{q+1} m)},
\]
and
\[
 \|\cQa^{+}(g,f)-\mathcal{Q}_{\alpha'}^+(g,f)\|_{L^1_v(\vv^q m)}\leq p(\alpha-\alpha')\|f\|_{W^{1,1}_v(\vv^{q+1}m)}\|g\|_{L^1_v(\vv^{q+1} m)}.
\]
Where $p(r)$ is an explicit polynomial converging to 0 if $r$ goes to 0.
\end{proposition}

\begin{proof}
The proof of this proposition follows exactly as in \cite[Proposition~3.2]{MischlerMouhotSelfSimilarity}, due to \eqref{eqn.WeightIneqAbsolute}. Let us just remark that in our case \cite[Lemma~3.3]{MischlerMouhotSelfSimilarity} reads as follows. For any $\delta>0$ and $\alpha\in (0,1)$ it holds that if $\sigma\in \S^2$ and $\sin^2 \xi\geq\delta$, where $\cos\xi=|\sigma\cdot(w/|w|)|$ and $w=v+v_*\neq 0$, then
\begin{equation}\label{eqn.IneqDeltaM}
 m(v')\leq e^b m^k(v)m^k(v_*),
\end{equation}
where $k=(1-\delta/160)^{\beta/2}$. \\

Moreover, from the proof of \cite[Lemma~3.3]{MischlerMouhotSelfSimilarity} we have that 
\[
 |v|^2\leq (1-\delta/160)(|v|^2+|v_*|^2).
\]
This, together with $\vv^{\beta}\leq 1+|v|^{\beta}$ concludes the proof of \eqref{eqn.IneqDeltaM}. 
\end{proof}

The next result is a reformulation of \cite[Proposition~ A.2]{BisiCanizoLodsUniqueness} for our weight function.

\begin{proposition}\label{prop.BoundH}
Set $m(v)=\exp(b\vv^{\beta})$ with $b>0$ and $\beta\in(0,1)$ and let
\[
 H(v_*)=\int k_{e}(v,v_*)\vv^q m(v)dv\quad \forall w\in\T^3.
\]
Then, there exist a constant $K=K(e,b,\beta)>0$ such that
\[
 H(v_*)\leq K(1+|v_*|^{1-\beta})\left\langle v_*\right\rangle^q m(v_*),
\]
for every $v_*\in\T^3$.
\end{proposition}

\begin{proof}
Consider $\hat{m}(v)=\exp(b|v|^\beta)$. It is easy to see that 
\begin{equation}\label{eqn.BCLWeightIneq}
 \hat{m}(v)\leq m(v)\leq e^b \hat{m}(b).
\end{equation}
Furthermore, there exist a constant $C'_{q,b}>0$ such that $\vv^q\leq C'_{q,b}(1+|v|^q)$, we have that  
\[
 H(v_*)\leq C_{q,b}\int k_{e}(v,v_*)(1+|v|^q) \hat{m}(v)dv,\\
\]
where $C_{q,b}=e^b C'_{q,b}$.\\

Let us recall that $k_e$ is given by 
\begin{align*}
 k_e&(v,v_*)\\
    &=\frac{C_e}{|v-v_*|}\exp\left\lbrace c_0\left((1+\mu)|v-v_*|+\frac{|v-u_0|^2-|v_*-u_0|^2}{|v-v_*|}\right)^2\right\rbrace,
\end{align*}
for some constants $C_e,\mu>0$ depending only on $e$ and where $c_0= 1/(8\theta_0)$. Here we will assume $u_0=0$ to simplify the computations, the general case follows in a similar manner. Taking into account that $\frac{|v|^2-|v_*|^2}{|v-v_*|}-|v-v_*|=2\frac{v-v_*}{|v-v_*|}\cdot v_*$, we can rewrite $k_e$ as
\[
 k_e(v,v_*)=\frac{C_e}{|v-v_*|}\exp\left\lbrace -c_0\left((2+\mu)|v-v_*|+2\frac{v-v_*}{|v-v_*|}\cdot v_*\right)^2\right\rbrace.\\
\]

Moreover, performing the change of variables $u=v-v_*$ and using spherical coordinates with $\rho=|u|$ and $\rho |v_*|y=u\cdot v_* $, one gets 
\[
 H(v_*)= C_0\int_A F(\rho,y)d\rho d y,
\]
with $A=[0,+\infty)\times[-1,1]$ and
\begin{align*}
 F(\rho,y)&=\left[1+(\rho^2+|v_*|^2+2\rho|v_*|y)^{q/2}\right]\rho\times\\
          &\exp\left\lbrace -c_0\left((2+\mu)\rho+2|v_*|y \right)^2+b\left(\rho^2+|v_*|^2+2\rho|v_*|y \right)^{q/2}\right\rbrace.
\end{align*}
 
Let us split $A$ into two regions $A=A_1\cup A_2$ where
\[
 A_1:=\{(\rho,y)\in A: 3|v_*|y\geq -2\rho\}\quand A_2:=A\backslash A_1.\\
\]

We first compute the integral over $A_1$. Notice that since $y\leq 1$ and $\beta\in(0,1)$ we get that for every $(\rho,y)\in A$
\begin{align*}
 \exp\left\lbrace b(\rho^2+|v_*|^2+2\rho|v_*|y)^{q/2}\right\rbrace&\leq \exp \left(b(\rho+|v_*|)^{\beta}\right)\\
   &\leq \exp\left(b\rho^{\beta}\right)\exp\left(b|v_*|^{\beta}\right).
\end{align*}
Moreover, since $(2+\mu)\rho+2|v_*|y\geq (\mu+2/3)\rho$ for any $(\rho,y)\in A_1$ we have
\begin{align}\label{eqn.IntA1}\nonumber
 \int_{A_1} F(\rho,y)d\rho dy &\leq C_{1,q}|v_*|^q e^{b|v_*|^{\beta}}\int_{0}^{+\infty}\rho\exp\left(-c_0(2/3+\mu)^2\rho^2+b\rho^{\beta}\right)dyd\rho\\\nonumber
  &+C_{2,q}e^{b|v_*|^{\beta}}\int_{0}^{+\infty}\rho^{q+1}\exp\left(-c_0(2/3+\mu)^2\rho^2+b\rho^{\beta}\right)dyd\rho\\\nonumber
  &\leq C_1 \left\langle v_*\right\rangle^q \hat{m}(v_*)\\
  &\leq C_1 \left\langle v_*\right\rangle^q m(v_*),
\end{align}
for some constant $C_1>0$, since the integrals are convergent.\\

On the other hand, for every  $(\rho,y)\in A_2$ we have
\[
 \rho^2+|v_*|^2+2\rho|v_*|y<|v_*|^2-\frac{1}{3}\rho^2\quand \rho<\frac{3}{2}|v_*|y.
\]
In that case, using the change of variables $z=(2+\mu)\rho+2|v_*|y$ we get 
\begin{align*}
 &\int_{A_2} F(\rho,y)d\rho dy \\
 &\leq C_{3,q}\left\langle v_*\right\rangle^q \int_{0}^{(3/2)|v_*|}\rho\exp\left(b(|v_*|^2-\frac{1}{3}\rho^2)^{\beta/2}\right)d\rho\cdot\int_{-1}^{1}\exp\left(-c_0((2+\mu)\rho+2|v_*|y)^2\right)dy\\
  &\leq C_{3,q}\cdot\frac{1}{2|v_*|}\cdot \left\langle v_*\right\rangle^q \int_{0}^{(3/2)|v_*|}\rho\exp\left(b(|v_*|^2-\frac{1}{3}\rho^2)^{\beta/2}\right)d\rho\cdot \int_{-\infty}^{+\infty}\exp\left(-c_0 z^2\right)dz\\
  &\leq C_{4,q}\cdot\frac{1}{2|v_*|}\cdot \left\langle v_*\right\rangle^q \int_{0}^{(3/2)|v_*|}\rho\exp\left(b(|v_*|^2-\frac{1}{3}\rho^2)^{\beta/2}\right)d\rho,
\end{align*}
since the integral over $z$ is finite. Finally, setting $w=|v_*|^2-\rho ^2/3$ we obtain,
\begin{align}\label{eqn.IntA2}\nonumber
 &\int_{A_2} F(\rho,y)d\rho dy \\\nonumber
 &\leq C_{4,q}\cdot\frac{3}{2|v_*|}\cdot \left\langle v_*\right\rangle^q \int_{|v_*|^2/4}^{|v_*|^2}\exp\left(b w^{\beta/2}\right)dw\\\nonumber
 &\leq C_{4,q}\cdot\frac{3}{2|v_*|}\cdot \left\langle v_*\right\rangle^q \int_{0}^{|v_*|^2}\exp\left(b w^{\beta/2}\right)dw\\\nonumber
 &\leq C_{4,q}\cdot\frac{3}{2|v_*|}\cdot \left\langle v_*\right\rangle^q \cdot \frac{2}{b\beta}|v_*|^{2-\beta}\exp(b|v_*|^{\beta})\\\nonumber
 &\leq C_2\left\langle v_*\right\rangle^q \hat{m}(v_*)|v_*|^{1-\beta}\\
 &\leq C_2\left\langle v_*\right\rangle^q m(v_*)|v_*|^{1-\beta},
\end{align}
for some positive constant $C_2$. Thus, putting together \eqref{eqn.IntA1} and \eqref{eqn.IntA2} we get the result.
\end{proof}

\section{Semigroup Generators}\label{sec.Semi}

The aim of this section is to prove that the operator $\cBad$ generates a $C_0$-semigroup in $L^1_{x,v}(\vv^q m)$ for any $q\geq 0$. Moreover, the same proof remains true in the spaces $E_j$ with $j=-1,0,1$ and $\cE$.\\

The proof presented here follows the one in \cite[Appendix~ C]{AlonsoBaglandLods}, with some adaptations due to the definition and splitting of the operator $\cBad$ presented here.\\

Let us recall the definition of the operator $\cBad$:
\[
 \cBad(h):=\cTaR(h)+\cLR(h)-\Sigma h-v\cdot \nabla_x h,
\]
where $\Sigma=\Na+\nu_e$, with domain $D(\cBad)=W^{1,1}_{x,v}(\vv^{q+1}m)$. Moreover, recall the definition of $\cTaR$ and $\cLR$
\begin{align*}
&\cTaR(h)=\cQaR^+(h,\Fa)+\cQaR^+(\Fa,h)-\cQaR^-(\Fa,h),\\
&\cLR(h)=\mathcal{Q}^+_{e,R}(h,\cM_0),
\end{align*}
where $\cQaR^+$, $\mathcal{Q}^+_{e,R}$ (resp. $\cQaR^-$) is the gain (resp. loss) part of the collision operator associated to the mollified collision kernel $(1-\Theta_{\delta})B$.\\

Consider the operator
\[
 A_0(h):=-\Sigma h-v\cdot\nabla_x h.
\] 
Notice that, by a similar argument as in \eqref{eqn.ColFreqIneq}, we have that there exist $\sigma_0,\sigma_1>0$ such that 
\[
 0<\sigma_0\leq \sigma_0\vv\leq\Sigma(v)\leq \sigma_1\vv,
\]
for every $v\in\R^3$. Therefore, the domain of $A_0$ coincides with the domain of $v\cdot\nabla_x h$ wich is $W^{1,1}_{x,v}(\vv^{q+1}m)$. It is easy to see that $A_0$ generates a $C_0$-semigroup $\{U(t):t\geq0\}$ given by
\[
 U(t)h(x,v):=e^{-\Sigma(v)t}h(x-tv,v),
\]
which satisfies 
\[
 \|U(t)h\|_{L^1_{x,v}(\vv^q m)}\leq e^{-\sigma_0t}\|h\|_{L^1_{x,v}(\vv^q m)}.
\]
In particular $\{U(t):t\geq0\}$ is a nonnegative contractive semigroup in $L^1_{x,v}(\vv^q m)$.\\

\begin{lemma}\label{lem.BoundRA}
For any $\alpha>0$, $q\geq 0$ and $\lambda>0$
 \begin{equation}\label{eqn.BoundRA0}
  \|R_{A_0}(\lambda)\|_{\sB(L^1_{x,v}(\vv^q m),L^1_{x,v}(\vv^{q+1} m))}\leq \frac{1}{\sigma_0},\\
 \end{equation}
and,
\begin{equation}\label{eqn.BoundRA01}
 \|R_{A_0}(\lambda)\|_{\sB(L^1_{x,v}(\vv^q m))}\leq \frac{1}{\lambda+\sigma_0}.\\
\end{equation}
\end{lemma}

\begin{proof}
Since, $\{U(t):t\geq0\}$ is a nonnegative semigroup $R_{A_0}(\lambda)$ is also nonnegative. Moreover, since the positive cone of $L^1_{x,v}(\vv^q m)$ is generating (i.e. every element in $L^1_{x,v}(\vv^q m)$ is the difference of two elements in the positive cone), it is enough to consider $h$ nonnegative.\\

Let $g=R_{A_0}(\lambda)h$. Thus, we have
\begin{align*}
 h&=(\lambda I-A_0)R_{A_0}(\lambda)h=(\lambda I-A_0)g\\
  &= (\lambda+\Sigma)g+v\cdot\nabla_x g.
\end{align*}
Multipling by the weight and integrating over $\R^3\times\T^3$ we get
\[
 \int_{\R^3\times\T^3}(\lambda+\Sigma(v))g(x,v)\vv^q m(v)dvdx=\|h\|_{L^1_{x,v}(\vv^{q} m)}.
\]
Hence, since $\sigma_0\vv\leq\Sigma(v)$ we have
\[
 \lambda\|g\|_{L^1_{x,v}(\vv^{q} m)}+\sigma_0\|g\|_{L^1_{x,v}(\vv^{q+1} m)}\leq \|h\|_{L^1_{x,v}(\vv^{q} m)},
\]
which concludes our proof.
\end{proof}

The next step is to split $\cTRR:=\cTaR+\cLR$ into positive and negative parts
\begin{align*}
&\cTRR^+(h):=\cQaR^+(h,\Fa)+\cQaR^+(\Fa,h)+\cLR(h),\\
&\cTRR^-(h):=\cQaR^-(\Fa,h).
\end{align*}

By a similar argument as the one presented in \cite[Proposition~ B.2]{CanizoLods} one has the following lemma:

\begin{lemma}\label{lem.BoundTR}
 For any $q\geq 0$ there exist $\kappa(\delta)$ going to 0 as $\delta\to0$, and such that for ever $h\in L^1_{x,v}(\vv^{q} m)$ 
 \[
  \|\cTRR^+(h)\|_{L^1_{x,v}(\vv^{q} m)}\leq\kappa(\delta)\|h\|_{L^1_{x,v}(\vv^{q+1} m)},
 \]
 and that $\cTRR^-$ is bounded in $L^1_{x,v}(\vv^{q} m)$.
\end{lemma}

Let us introduce the operator $A_1:=A_0+\cTRR^+$. We want to prove that $\lambda I- A_1$ is invertible. In order to do this, notice that
\[
 \lambda I-A_1=(\lambda I-A_0)( I-R_{A_0}\cTRR^+),
\] 
so it is enought to see that $ I-R_{A_0}\cTRR^+$ is invertible. By the lemmas \ref{lem.BoundRA} and \ref{lem.BoundTR} we have that for every $\lambda>0$
\begin{align}\label{eqn.BoundA1}
 \|\cTRR^+R_{A_0}(\lambda) h\|_{L^1_{x,v}(\vv^{q} m)}&\leq \kappa(\delta)\|R_{A_0}(\lambda) h\|_{L^1_{x,v}(\vv^{q+1} m)}\\\nonumber
 &\leq \frac{\kappa(\delta)}{\sigma_0}\|h\|_{L^1_{x,v}(\vv^{q} m)}.
\end{align}
Thus, taking $\delta$ small enough we have $\kappa(\delta)<\sigma_0$, so one deduces that $\R^+\subset\rho(A_1)$ and, for every $\lambda>0$, $ I-R_{A_0}\cTRR^+$ is invertible. Moreover, by the Neumann
series (see \cite[Remark~ 2.34]{BanasiakArlotti}) one has 
\[
 R_{A_1}(\lambda)=R_{A_0}\sum_{j=0}^{\infty}\left(\cTRR^+ R_{A_0}(\lambda)\right)^j.
\]

Therefore, according to \eqref{eqn.BoundRA01}
\[
 \lim_{\lambda\to\infty}\|R_{A_1}(\lambda)\|_{\sB(L^1_{x,v}(\vv^{q} m))}\leq \lim_{\lambda\to\infty}\frac{1}{\lambda+\sigma_0}=0.\\
\]

Finally, notice that $\cBad=A_1-\cTRR^-$. Since $\cTRR^-$ is bounded, one can take $\lambda$ large enough such that
\[
 \|\cTRR^-\|_{\sB(L^1_{x,v}(\vv^{q} m))}\|R_{A_1}(\lambda)\|_{\sB(L^1_{x,v}(\vv^{q} m))}<1.
\]

Hence, one deduces that $\lambda I-\cBad$ is invertible for $\lambda$ large enough. This, together with the hypo-dissipativity ensures that $\cBad$ generates a $C_0$-semigroup due to Lumer-Phillips theorem (see \cite[Theorem~ 3.19]{BanasiakArlotti} or \cite[Theorem~ 4.3]{Pazy}).

\section{Spectral Theorems}

In this section we present a more abstract theorem regarding enlargement of the functional space semigroup decay. More specifically:\\

\begin{theorem}{\cite[Theorem~2.13]{GualdaniMischlerMouhot}}\label{thm.Enlargement}
Let $E',\cE'$ be two Banach spaces with $E'\subset\cE'$ dense with continuous embedding, and consider $L\in\mathcal{C}(E')$, $\cL\in\mathcal{C}(\cE')$ with $\left.\cL\right|_{E'}=L$ and $a\in\R$. Assume
\begin{enumerate}[(A)]
\item \label{it.SGGualdani1} $L$ generates a semigroup $e^{tL}$ in $E'$, $L-a$ is hypodissipative on $Range(\id-\Pi_{L,a})$ and 
      \[
       \Sigma(L)\cap\Delta_a:=\{\xi_1,...,\xi_k\}\subset \Sigma_d(L).
      \]
\item \label{it.SGGualdani2} There exist $\cA,\cB\in\mathcal{C}(\cE')$ such that $\cL=\cA+\cB$, $\left.\cA\right|_{E'}=A$, $\left.\cB\right|_{E'}=B$, some $n\geq1$ and $C_a>0$ such that
\begin{enumerate}[(B.1)]
\item \label{it.SGGualdani21} $\cB-a$ is hypodissipative in $\cE'$,
\item \label{it.SGGualdani22} $\cA\in\sB(\cE')$ and $A\in\sB(E')$,
\item \label{it.SGGualdani23} $T_n:=(\cA S_{\cB})^{(*n)}$ satisfies $\|T_n(t)\|_{\sB(\cE',E')}\leq C_a e^{at}$.\\
\end{enumerate}
Then $\cL$ is hypodissipative in $\cE'$ with
\[
 \|S_{\cL}(t)-\sum_{j=1}^k S_L(t)\Pi_{\cL,\xi_j}\|_{\sB(\cE')}\leq C'_a t^ne^{at},
\]
for all $t\geq 0$ and for some $C'_a>0$. 
\end{enumerate}
\end{theorem} 

Actually, the assumption \ref{it.SGGualdani23} follows from \cite[Lemma~ 2.17]{GualdaniMischlerMouhot} which yields an estimate on the norms $\|T_n\|_{\sB(E_j ,E_{j+1})}$ for $j =-1, 0$:

\begin{lemma}\label{lem.Enlargement}
Let $X,Y$ be two Banach space with $X\subset Y$ dense with continuous embedding, and consider $L\in \sB(X)$,  $\cL\in \sB(Y)$ such that $\left.\cL\right|_X = L$, $\cL=\cA+\cB$ and $a\in\R$. We assume that there exist some intermediate spaces
\[
 X=\cE_J\subset\cE_{J-1}\subset\cdots\cE_2\subset\cE_1=Y,
\]
with $J\geq2$, such that if we denote $\cA_j=\left.\cA\right|_{\cE_j}$ and $\cB_j=\left.\cB\right|_{\cE_j}$ 
\begin{enumerate}
\item $(B_j-a)$ is hypodissipative and $\cA_j$ is bounded on $\cE_j$ for $j=1,..,J$.
\item  There are some constants $l\in\N^*$, $C\geq1$, $K\in\R$, $\gamma\in[0,1)$ such that for all $t\geq0$ for $j=1,..,J-1$
\[
 \|T_{l}(t)\|_{\sB(\cE_j,\cE_{j+1})}\leq C\frac{e^{Kt}}{t^{\gamma}}.
\]
\end{enumerate}
Then for any $a'>a$, there exist some constructive constants $n\in\N$, $C_{a'}\geq1$  such that for all $t\geq0$ 
\[
 \|T_{n}(t)\|_{\sB(\cE_j,\cE_{j+1})}\leq C_{a'}e^{a't}.
\]
\end{lemma}

Furthermore, we state a quantitative spectral mapping theorem. A proof for this result can be found in \cite[Proposition~ 2.20]{Tristani}. A more general version of this theorem can be found in \cite{MischlerScher}.

\begin{theorem}\label{thm.SpectralMappingThm}
Consider a Banach space $X$ and an operator $\Lambda\in\sC(X)$ so that $\Lambda=\cA+\cB$ where $\cA\in\sB$ and $\cB-a$ is hypodissipative on $X$ for some $a\in\R$. We assume furthermore that there exist a family $X_j$, $1\leq j\leq m$, $m\geq 2$ of intermediate spaces such that
\[
 X_m\subset D(\Lambda^2)\subset X_{m-1}\subset\cdots X_2\subset X_1=X,
\]
and a family of operators $\Lambda_j$, $\cA_j$, $\cB_j\in\sC(X_j)$ such that
\[
 \Lambda_j=\cA_j+\cB_j, \quad \Lambda_j=\left.\Lambda\right|_{X_j}, \quad \cA_j=\left.\cA\right|_{X_j}, \quad \cB_j=\left.B\right|_{X_j},
\]
and that it holds
\begin{enumerate}[(i)]
\item \label{it.SpectralMapThmI} $( B_j-a)$ is hypodissipative on $X_j$;
\item \label{it.SpectralMapThmII} $\cA_j\in\sB(X_j)$;
\item \label{it.SpectralMapThmIII} there exist $n\in\N$ such that $T_n(t):=(\cA S_{\cB}(t))^{*n}$ satisfies \[\|T_n(t)\|_{\sB(X,X_m)}\leq Ce^{at}.\]
\end{enumerate}
Hence, the following localization of the principal part of the spectrum
\begin{enumerate}[(A)]
\item \label{it.SpectralMapThm1} There are some distinct complex numbers $\xi_1,...,\xi_k\in\Delta_a$, $k\in\N$ such that 
\[
 \Sigma(\Lambda)\cap\Delta_a=\{\xi_1,...,\xi_k\}\subset\Sigma_d(\Lambda);
\]
\end{enumerate}
implies the following quantitative growth estimate on the semigroup:
\begin{enumerate}[(B)]
\item \label{it.SpectralMapThm2} for any $a'\in(a,\infty)\backslash\{\Re\xi_j,j=1,...,k\}$, there exist some constructive constant $C_{a'}>0$ such that for every $t\geq 0$
\[
 \left\|S_{\Lambda}(t)-\sum_{j=1}^k S_{\Lambda}(t)\Pi_{\Lambda,\xi_j}\right\|_{\sB(X)}\leq C_{a'} e^{a't}.
\]
\end{enumerate}
\end{theorem}


\bibliography{bibliography}
\nocite{AlonsoLods,AlonsoGamba, AlonsoGambaTaskovic}
\bibliographystyle{plain}

\end{document}